\documentclass[12pt,twoside]{amsart}
\usepackage{latexsym,amsmath,amsopn,amssymb,amsthm,amsfonts}
\usepackage[T2A]{fontenc}
\usepackage[cp1251]{inputenc}
\usepackage[english]{babel}
\usepackage{graphicx}

\newtheorem{theorem}{Theorem}

\newtheorem{proposition}{Proposition}

\newtheorem{remark}{Remark}

\newtheorem{corollary}{Corollary}

\newtheorem{lemma}{Lemma}

\DeclareMathOperator{\ad}{ad}

{
\newcounter{table_counter}	
\newcommand{\Table}{\refstepcounter{table_counter}\newline\newline{\it Таблица \arabic{table_counter}}.}
	
\begin{document}
\vspace{0.5cm}
\title[Abnormal extremals of sub-Finsler  $q$-metrics on Lie groups]{Abnormal extremals of left-invariant sub-Finsler quasimetrics on  four-dimensional Lie groups with three-dimensional generating distributions}
\author{V.~N.~Berestovskii, I.~A.~Zubareva}
\thanks{The first author is supported by the Mathematical Center in Akademgorodok under Agreement No. 075-15-2019-1613 with the Ministry of Science and Higher Education of the Russian Federation.
The second author is supported by the program of fundamental scientific investigations of SD RAS № I.1.1., project № 0314--2019--0004.}
\address{Sobolev Institute of Mathematics,\newline
	4 Koptyug Av., Novosibirsk, 630090, Russia}
\email{valeraberestovskii@gmail.com}
\address{Sobolev Institute of Mathematics,\newline
	13 Pevtsov Str., Omsk, 644099, Russia}
\email{i\_gribanova@mail.ru}	
	
\begin{abstract}
We find three-dimensional subspaces of four-dimensional connected Lie algebras, generating these algebras, and 
abnormal extremals on connected Lie groups with these Lie algebras and with left-invariant sub-Finsler quasimetrics defined by seminorms on such subspaces. 
In terms of the structure constants of Lie algebras and dual seminorms, we establish a criterion for the strong abnormality of these extremals. 	
 \vspace{2mm}

\vspace{2mm}

\end{abstract}

\maketitle
		
\section*{Introduction}

In 	\cite{Ber1}	was indicated that the shortest arcs of a left-invariant sub-Finsler metric $d$ on a 
connected Lie group $G$, defined by a left--invariant bracket generating distribution $D$ and a norm $F$ on 
$D(e)=\mathfrak{q}\subset TG_e=\mathfrak{g}$, are solutions of a left-invariant time-optimal problem for the 
closed unit ball with zero center of the normed vector space $(D(e),F)$  as the control region.
The  distribution $D$ is bracket generating if and only if the subspace $\mathfrak{q}$ generates 
the Lie algebra $(\mathfrak{g},[\cdot,\cdot])$ by the Lie bracket $[\cdot,\cdot].$ 
Moreover, the statements about shortest arcs are also true for a pair  $(D(e),F)$ with a seminorm $F$ which satisfies $F(u)>0$ for $0\neq u\in D(e)$ and defines a left-invariant sub-Finsler quasimetric $d$ on $G$. 

The Pontryagin Maximum Principle (PMP) \cite{PBGM} gives some necessary conditions for solutions to the time-optimal control problem. 
An extremal is a curve in $G$ which is parameterized by the arclength and satisfies the PMP. 
		
An extremal can be normal or abnormal. Some extremals can be both normal and abnormal with respect to different covector functions in PMP; such extremals are called {\sl  nonstronly abnormal}. An abnormal extremal that is not nonstronglyy abnormal is {\sl called strongly abnormal}.
		
In this paper, we solve the search problem for abnormal extremals on four-dimensional connected Lie groups with a left-invariant sub-Finsler quasimetric
defined by a seminorm on a three-dimensional subspace $\mathfrak{q}$ of the Lie algebra of each such group generating the Lie algebra by the Lie bracket $[\cdot,\cdot].$ We establish a criterion for the nonstrong abnormality of these extremals which also allows us to formulate the criterion of their strong abnormality. Since the quasimetric is left-invariant, we can assume that extremals issue from the unit of the group. Each of these abnormal extremals is some one-parametric subgroup of the Lie group. Moreover, the abnormal extremal in $(G,d)$ with  a left-invariant sub-Riemannian metric $d$ is a geodesic, i.e. its sufficiently small segments are the shortest arcs, if and only if this extremal is nonstrongly abnormal.

A four-dimensional Lie group $G$ has abnormal extremals for all seminorms $F$ of the above-indicated form on $\mathfrak{q}\subset\mathfrak{g}$
if  $\dim(\mathfrak{q})=3$ and $\mathfrak{q}$ generates $\mathfrak{g}$. In this paper, we consider such Lie groups $G$ and $\mathfrak{q}\subset\mathfrak{g}$. On the ground of papers \cite{Ber88} and \cite{Mil}, for every dimension $n > 2$ there exist exactly two real Lie algebras that do not admit $(n-1)$-dimensional subspaces generating these Lie algebras.
For the remaining Lie algebras, we find up to their
automorphisms the number of three-dimensional subspaces generating these algebras.

All main results of the paper are obtained by using results of 
\cite{Ber1}, \cite{PBGM}, \cite{Ber}, \cite{BerZub}, \cite{M}, \cite{BR}, 
\cite{Patera}, \cite{BerZub21}.
In the last of the listed papers we considered 	a similar problem for two-dimensional bracket generating subspaces and gave  some detailed commentary on the four-dimensional Lie algebras.

\section{General algebraic results}
\label{alg}
		
The following proposition is obvious.		
		
\begin{proposition}
\label{gen}	
A $(n-1)$-dimensional subspace $\mathfrak{q}$ of a real Lie algebra $\mathfrak{g}$ of the dimension $n > 2$ generates
this algebra (by the Lie bracket $[\cdot,\cdot]$) if and only if the subspace $\mathfrak{q}$ isn't a subalgebra of the Lie algebra $\mathfrak{g}.$
\end{proposition}
		
\begin{proposition}
\label{3sa}
A $(n-1)$-dimensional subspace $\mathfrak{s}$ of a $n$-dimensional  Lie algebra $\mathfrak{g},$ $n\geq 3$, is its subalgebra if the
dimension of the intersection of $\mathfrak{s}$ with its normalizer $\mathfrak{N}(\mathfrak{s})$ in the Lie algebra $\mathfrak{g}$ is at least $n-2$.
\end{proposition}
		
\begin{proof}
This statement follows from the fact that the codimension of the mentioned intersection	in $\mathfrak{s}$ is at most 1.
\end{proof}
		
Propositions \ref{gen} and \ref{3sa} imply	
		
\begin{proposition}
\label{por}
If a $(n-1)$-dimensional subspace $\mathfrak{q}$ generates a $n$-dimensional real Lie algebra $\mathfrak{g},$ $n\geq 3$, then $\dim(\mathfrak{q}\cap \mathfrak{N}(\mathfrak{q}))\leq n-3$
and $\mathfrak{q}$ isn't a subalgebra of the Lie algebra $\mathfrak{g}$.
Consequently, the dimension of the intersection  of $\mathfrak{q}$ with any $(n-1)$--dimensional subalgebra of the Lie algebra $\mathfrak{g}$ (respectively, with any its $(n-2)$-dimensional ideal)
is equal to $n-2$ (respectively, $n-3$).
\end{proposition}

The following proposition follows from the proof of Theorem 4 in \cite{Ber88} based on \cite{Mil}.
		
\begin{proposition}
\label{gen1}	
A real Lie algebra $\mathfrak{g}$ of dimension $n > 2$ has no  $(n-1)$-dimensional bracket generating subspace if and only if
$\mathfrak{g}$ is commutative or $\mathfrak{g}$ includes a commutative $(n-1)$-dimensional ideal  $I$ and an
element $z$ such that the operator $\ad z$ acts identically on $I.$ 
\end{proposition}

The following proposition  was proved in \cite{BerZub21}.
		
\begin{proposition}
\label{threed}
A four-dimensional connected Lie group  $G$ with a Lie algebra $\mathfrak{g}$ and a three-dimensional generating subspace $\mathfrak{q}\subset\mathfrak{g}$
has abnormal extremals (for arbitrary left-invariant quasimetric $d$ on $G$ defined by a seminorm $F$ on $\mathfrak{q})$ if and only if
$\mathfrak{q}_1=\mathfrak{q}\cap\mathfrak{N}(\mathfrak{q})\neq\{0\},$ where $\mathfrak{N}(\mathfrak{q})$ is the normalizer of $\mathfrak{q}$ in
$\mathfrak{g}$. Furthermore, $\dim(\mathfrak{q}_1)=1$ and every one-parameter subgroup $g=g(t)=\exp(tX),$ where $X\in\mathfrak{q}_1,$ $F(X)=1,$ is an abnormal extremal for $(G,d);$ there is no other abnormal extremal with origin $e\in G$. Moreover, the extremal $g$ is strongly abnormal 
(nonstrongly abnormal) for any quasimetric $d$ if and only if $\mathfrak{q}_1\subset [\mathfrak{q}_1,\mathfrak{q}]$ 
(respectively, $\mathfrak{q}_1=\mathfrak{q}\cap\mathfrak{C}(\mathfrak{q}),$ where $\mathfrak{C}(\mathfrak{q})$ is the centralizer of 
$\mathfrak{q}$ in $\mathfrak{g}$.)
\end{proposition}
		
\begin{proposition}
\label{cq}	
$\mathfrak{C}(\mathfrak{q})=\mathfrak{C}(\mathfrak{g}).$ 
\end{proposition}
		
\begin{proof}
This statement is a consequence of the Jacobi identity and the fact that $\mathfrak{q}$ generates $\mathfrak{g}.$		
\end{proof}
		
\begin{lemma} 
\label{base}	
Let  $(\mathfrak{g},[\cdot,\cdot])$ be a four-dimensional real Lie algebra, $\mathfrak{q}\subset \mathfrak{g}$ be a three-dimensional subspace generating $\mathfrak{g}$ by the Lie bracket  $[\cdot,\cdot]$. Then there exists (the unique) one-dimensional subspace  $\mathfrak{q}_1=\mathfrak{q}\cap\mathfrak{N}(\mathfrak{q}),$ where $\mathfrak{N}(\mathfrak{q})$ is the normalizer of $\mathfrak{q}$ in $\mathfrak{g}.$
Furthermore, $[\mathfrak{p},\mathfrak{p}]\not\subset \mathfrak{q}$ and $\dim([\mathfrak{p},\mathfrak{p}])=1$ for any two-dimensional subspace $\mathfrak{p}\subset \mathfrak{q}$ such that $\mathfrak{p}\cap \mathfrak{q}_1=\{0\}.$ In other words, for any linearly independent vectors $e_1,e_3\in\mathfrak{p}$ and a nonzero vector $e_2\in\mathfrak{q}_1$, the vectors $e_1,e_2,e_3,e_4=[e_1,e_3]$ constitute a basis of the Lie algebra $\mathfrak{g}$ and for this basis
\begin{equation}
\label{fghl}
C_{13}^1=C_{13}^2=C_{13}^3=0,\,\,\,C_{13}^4=1,\,\,\,C_{12}^4=C_{23}^4=0.
\end{equation}
\end{lemma}	
		
\begin{proof}
Proposition \ref{gen} and the condition that $\mathfrak{q}$ generates $\mathfrak{g}$ imply that  $[\mathfrak{q},\mathfrak{q}]\not\subset \mathfrak{q},$ i.e. there exist some linearly independent vectors $e_1,e_3\in \mathfrak{q}$ such that $[e_1,e_3]\not\in\mathfrak{q}.$
Then for any vector $e_2\in\mathfrak{q}$ linearly independent with vectors
$e_1,e_3$ the vectors $e_1,e_2,e_3,e_4=[e_1,e_3]$ constitute a basis of $\mathfrak{g}.$ 
If at least one of the structure constants $C_{12}^4$, $C_{23}^4$ in this basis is nonzero then, 
replacing $e_2$ with $e_2-C_{23}^4e_1-C_{12}^4e_3$ and again denoting the last vector by
$e_2$, we get $[\mathfrak{p}_1,\mathfrak{q}]\subset\mathfrak{q}$ for a one-dimensional subspace 
$\mathfrak{p}_1\subset\mathfrak{q}$ spanned by the vector $e_2$. Since $\mathfrak{q}$ generates $\mathfrak{g},$ then $[\mathfrak{p},\mathfrak{p}]\not\subset \mathfrak{q}$ and $\dim([\mathfrak{p},\mathfrak{p}])=1$ for any two-dimensional subspace $\mathfrak{p}\subset \mathfrak{q}$ such that $\mathfrak{p}\cap \mathfrak{p}_1=\{0\}.$ It follows that  
$\mathfrak{p}_1=\mathfrak{q}\cap\mathfrak{N}(\mathfrak{q})=\mathfrak{q}_1.$
\end{proof}
		
\begin{corollary}
\label{act}	
Every four-dimensional connected Lie $G$ with a Lie algebra $\mathfrak{g}$ and a three-dimensional generating subspace $\mathfrak{q}\subset\mathfrak{g}$ has
abnormal extremals (for arbitrary left-invariant quasimetric $d$ on $G$ defined by some seminorm $F$ on $\mathfrak{q})$.
If $\mathfrak{q}$ contains a one-dimensional central (respectively, noncentral) ideal of the Lie algebra $\mathfrak{g}$, then every abnormal extremal is nonstrongly (respectively, strongly) abnormal.
\end{corollary}
		
\begin{proof}
It follows from Proposition \ref{threed} and Lemma \ref{base}.	
\end{proof}
		
We say that a basis $(e_1,e_2,e_3)$	of the subspace $\mathfrak{q}$ is {\it from Lemma \ref{base}}, if $(e_1,e_2,e_3,e_4:=[e_1,e_3])$ is a basis of the Lie algebra $\mathfrak{g}$ and (\ref{fghl}) is satisfied.
		
\begin{proposition}
\label{const}	
For the basis $(e_1,e_2,e_3,e_4)$ of the algebra  $\mathfrak{g}$ from Lemma \ref{base}, we have
$$C_{24}^1=C_{12}^1C_{23}^2-C_{12}^2C_{23}^1,\quad C_{24}^2=0,\quad C_{24}^3=C_{12}^3C_{23}^2-C_{12}^2C_{23}^3,\quad C_{24}^4=C_{23}^3-C_{12}^1.$$	
\end{proposition}
		
\begin{proof}
Owing to (\ref{fghl}), the Jacobi identity	
$$[e_1,[e_2,e_3]]+[e_2,[e_3,e_1]]+[e_3,[e_1,e_2]]=0$$ 
is equivalent to the equality $[e_1,[e_2,e_3]]-[e_2,e_4]+[e_3,[e_1,e_2]]=0$; i.e.	
$$0=C_{23}^2[e_1,e_2]+C_{23}^3e_4-[e_2,e_4]-C_{12}^1e_4-C_{12}^2[e_2,e_3]=\left(C_{23}^3-C_{24}^4-C_{12}^1\right)e_4+$$
$$\left(C_{12}^1C_{23}^2-C_{12}^2C_{23}^1-C_{24}^1\right)e_1-C_{24}^2e_2+\left(C_{12}^3C_{23}^2-C_{12}^2C_{23}^3-C_{24}^3\right)e_3.$$
This yields Proposition \ref{const}. 
\end{proof}	
		
\section{Criteria for the (non)strong abnormality of the extremal}
\label{osn} 
	
It was indicated in	\cite{Ber1} that parametrized by arclength shortest curves of a left-invariant sub-Finsler
metric $d$ on any connected Lie group $G$, defined by a left-invariant bracket generating distribution $D$ and
a norm  $F$ on $D(e)$, coincide with solutions to the time-optimal control problem for the system
\begin{equation}
\label{a3}
\dot{g}(t)=dl_{g(t)}(u(t)),\quad u(t)\in U,
\end{equation}
with a measurable control $u=u(t)$. Here $l_g(h)=gh$, the control domain is the unit ball $U=\{u\in D(e)\,|F(u)\leq 1\},$
while $D$ is bracket generating if and only if the corresponding subspace  $\mathfrak{q}:= D(e)\subset \mathfrak{g}$
satisfies the hypotheses of Lemma \ref{base}. It is clear that every parametrized
by arclength shortest curve $g(t),\,\,0\leq t\leq a,$ satisfies (\ref{a3}) and $F(u(t))=1$ for almost all $t\in [0,a]$.

These statements are true also for the case when $d$ is a quasimetric (respectively, $F$ is a seminorm on $D(e)$ such that $F(u)>0$ for $0\neq u\in D(e)$). 
		
Each segment of the shortest curve in $(G,d)$ is a shortest path, each open ball of sufficiently small positive radius in $(G,d)$ is
diffeomorphic to a region of the Euclidean space and each shortest path, joining any point of the ball with its center, lies in this ball.
Therefore, due to the Pontryagin Maximum Principle (PMP) \cite{PBGM} for the time-optimality of a control $u(t)$ and corresponding trajectory $g(t),$  $t\in [0,a],$ necessary exists a nowhere zero absolutely continuous covector function
$\psi(t)\in T^{\ast}_{g(t)}G$ such that for almost all $t\in [0,a]$ the function $\mathcal{H}(g(t);\psi(t);u)=\psi(t)(dl_{g(t)}(u))$
of $u\in U$ attains a maximum at the point $u(t)$: 
\begin{equation}
\label{m}
M(t)=\psi(t)(dl_{g(t)}(u(t)))=\max\limits_{u\in U}\psi(t)(dl_{g(t)}(u)).
\end{equation}
Moreover, $M(t)\equiv M\geq 0,$  $t\in [a,b].$ 
		
By an {\it extremal} we will mean a parametrized curve $g(t)$ in $G$ with a maximally admissible connected domain
$\Omega\subset\mathbb{R}$ which satisfies the PMP, conditions (\ref{a3}), and $F(u(t))=1$ with a measurable function $u(t)$ almost everywhere on the maximal subset in $\Omega.$ In the case $M=0$ (respectively, $M>0$)
an extremal is called {\it abnormal} (respectively, {\it normal}). In the normal case, proportionally changing
$\psi=\psi(t)$, $t\in\mathbb{R}$, if it is necessary, we can assume that $M=1$.
		
The following proposition is an immediate consequence of Lemma \ref{base}, Corollary 
\ref{act}, and Proposition \ref{threed}. But we shall give here its independent proof because some its details are needed further to establish criteria for strong and nonstrong abnormality of the extremals from Proposition \ref{prop1}.  
		
\begin{proposition}
\label{prop1}
Every four-dimensional connected Lie group $G$ with Lie algebra  $\mathfrak{g}$
and a three-dimensional generating subspace $\mathfrak{q}\subset\mathfrak{g}$ has abnormal extremals 
(for an arbitrary left-invariant quasimetric $d$ on $G$ defined by a seminorm $F$ on $\mathfrak{q}$). 
Each abnormal extremal in $(G,d)$ is one of the two one-parameter subgroups
\begin{equation}
\label{dif}
g(t)=\exp\left(\frac{ste_2}{F(se_2)}\right),\quad t\in \mathbb{R},\,\,s=\pm 1,
\end{equation}
or its left shift on $(G,d)$. 
\end{proposition}
		
\begin{proof}
We can consider the covector function $\psi(t)\in T^{\ast}_{g(t)}G$ from PMP 
as a left-invariant $1$-form on $(G,\cdot)$ and so naturally identify the latter with a covector function
$\psi(t)\in\mathfrak{g}^{\ast}=T^{\ast}_eG$.

In \cite{Ber}, \cite{BerZub} for an extremal $g(t)\in G$, are proved the following relations satisfying for almost all
$t$ in the domain:
\begin{equation}
\label{difur}
\dot{g}(t)=dl_{g(t)}(u(t)),\,\, (\psi(t)(v))'=\psi(t)([u(t),v]),\,\,u(t), v\in \mathfrak{g},\quad F(u(t))=1. 
\end{equation}
			
Omitting for brevity the parameter $t,$ we can write the second equation in (\ref{difur}) as
$\psi'(v)=\psi([u,v]).$ In particular, for $\psi_i:=\psi(e_i),$ $i=1,2,3,4,$ we have
\begin{equation}
\label{ssw}
\psi_i'= \psi([u,e_i]).
\end{equation} 
Set $u=u_1e_1+u_2e_2+u_3e_3\in U$. We get from (\ref{ssw}) and (\ref{fghl}):
\begin{equation}
\label{psi1}
\psi_1'=\psi(-u_2[e_1,e_2]-u_3e_4)=-u_2\sum\limits_{k=1}^3C_{12}^k\psi_k-u_3\psi_4,
\end{equation} 
\begin{equation}
\label{psi2}
\psi_2'=\psi(u_1[e_1,e_2]-u_3[e_2,e_3])=\sum\limits_{k=1}^3\left(u_1C_{12}^k-u_3C_{23}^k\right)\psi_k,
\end{equation}
\begin{equation}
\label{psi3}
\psi_3'=\psi(u_1e_4+u_2[e_2,e_3])=u_1\psi_4+u_2\sum\limits_{k=1}^3C_{23}^k\psi_k,
\end{equation}
\begin{equation}
\label{psi4}
\psi_4'=\psi(u_1[e_1,e_4]+u_2[e_2,e_4]+u_3[e_3,e_4])=\sum\limits_{k=1}^4\left(u_1C_{14}^k+u_2C_{24}^k+u_3C_{34}^k\right)\psi_k.
\end{equation}
Clearly, in abnormal case it must be $\psi_1=\psi_2=\psi_3\equiv 0.$ 
Then (\ref{psi1})---(\ref{psi4}), the condition $\psi_4\neq 0$ and the equation $F(u)=1$ imply that 
\begin{equation}
\label{u13}
u_1=u_3=0,\quad u_2=s/F(se_2),\quad s=\pm 1.
\end{equation}
It follows from (\ref{u13}), Proposition \ref{const} and (\ref{psi4}) that the function
\begin{equation}
\label{ ppsi4}
\psi_4(t)=\varphi_4\exp\left(\frac{C_{24}^4st}{F(se_2)}\right)=\varphi_4\exp\left(\frac{\left(C_{23}^3-C_{12}^1\right)st}{F(se_2)}\right),\,\,s=\pm 1,
\end{equation}
is a solution of equation (\ref{psi4}) with the initial condition $\psi_4(0)=\varphi_4\neq 0$. 
Obviously, it is possible to find $u(t),$ $\psi(t)$ by the above formulas for all $t\in\mathbb{R}.$
			
Now Proposition \ref{prop1} follows from (\ref{u13}) and the first equation in (\ref{difur}).
\end{proof}
		
Below  $F(u_1,u_2,u_3):=F(u),$ $F_U$ is the Minkowski supporting function of the body $U$:
$$F_U(x,y,z)=\max\limits_{(u_1,u_2,u_3)\in U}\left(xu_1+yu_2+zu_3\right).$$		
		
\begin{theorem}
\label{main1}
Abnormal extremal (\ref{dif}) of a four-dimensional connected Lie group  $G$ with a Lie algebra $(\mathfrak{g},[\cdot,\cdot])$ and a left-invariant
quasimetric $d$ defined by a seminorm  $F$ on  a three-dimensional subspace $\mathfrak{q}$, generating $\mathfrak{g}$, with a basis $(e_1,e_2,e_3=[e_1,e_2])$ from Lemma \ref{base} is nonstrongly abnormal if and only if one of the conditions is fulfilled:
$$C_{23}^1=C_{23}^2=0\quad \mbox{and for}\quad s=\pm 1,\quad \exists\quad  k(s)\in\mathbb{R}:  F_U(k(s),s,0)=1/F(0,s,0);$$  
$$C_{23}^1\neq 0\quad \mbox{and}\quad F_U(-C_{23}^2s/C_{23}^1,s,0)=1/F(0,s,0),\quad s=\pm 1.$$ 
\end{theorem}

\begin{proof} 
{\sl Necessity.} Assume that abnormal extremal (\ref{dif}) is nonstrongly abnormal. Then there exists a real-analytic covector function 
$\psi(t)$ which is a solution of system (\ref{psi1})\,--\,(\ref{psi4}), and $\psi(t)(u(t))=F_U(\psi_1(t),\psi_2(t),\psi_3(t))=1$
for almost all $t\in\mathbb{R}$. This and (\ref{u13}) imply that
\begin{equation}
\label{fp}
\psi_2(t)=1/u_2,\quad F_U(\psi_1(t),1/u_2,\psi_3(t))=1,	
\end{equation} 
and the points $\left(\psi_1(t),1/u_2,\psi_3(t)\right)$, $t\in\mathbb{R}$, are dual for the point $(0,u_2,0)$. Therefore, the ranges of the functions $\psi_1(t)$, $\psi_3(t)$, $t\in\mathbb{R}$, are segments (degenerating to a point if the function $F$ is differentiable at $(0,u_2,0)$)
because the body $U^{\ast}$ , dual to $U$, is convex and bounded.
	
We get from (\ref{psi1})--(\ref{psi4}), (\ref{u13}), $e_3=[e_1,e_2]$, and Proposition \ref{const} that 
\begin{equation}
\label{system}
\psi_1'=-u_2\psi_3,\,\,\psi_2\equiv\varphi_2,\,\,\psi_3'=u_2\sum\limits_{k=1}^3C_{23}^k\psi_k,\,\, \psi_4'=u_2\left(C_{24}^1\psi_1+\sum\limits_{k=3}^4C_{24}^k\psi_k\right).
\end{equation}
	
The first equality in (\ref{fp}) and (\ref{system}) imply
\begin{equation}
\label{equat}
\psi_1^{\prime\prime}-u_2C_{23}^3\psi_1^{\prime}+u_2^2C_{23}^1\psi_1+u_2C_{23}^2=0.
\end{equation} 
	
Assume that $C_{23}^1=0$. Then general solution to the equation (\ref{equat}) has a form
$$\psi_1(t)=\left\{\begin{array}{lr}
A_2e^{C_{23}^3u_2t}+C_{23}^2t/C_{23}^3+A_1,\quad\text{if } C_{23}^3\neq 0, \\
-\frac{1}{2}C_{23}^2u_2t^2+A_2t+A_1,\quad\text{if }C_{23}^3=0, \\
\end{array}\right.$$
where $A_1,\,A_2$ are arbitrary reals. 
Since $\psi_1(t)$, $t\in\mathbb{R}$, is bounded, this implies that $C_{23}^2=0$ and $\psi_1(t)=A_1$, $A_1\in\mathbb{R}$. Then $\psi_3(t)=0$ on the ground of the first equation in (\ref{system}) and, taking into account (\ref{u13}), (\ref{fp}), there exists a real number $k(s)$ such that $F_U(k(s),s,0)=1/F(0,s,0),$ $s=\pm 1.$ Hence, the supporting plane of the body
$U$ at the intersection point $\left(0,s/F(se_2),0\right)$ by the axis $Ou_2$ is parallel to the axis $Ou_3$.	
	
Now assume  that  $C_{23}^1\neq 0$. Let us set $B=\left(C_{23}^3\right)^2-4C_{23}^1$. Then general solution to the equation (\ref{equat}) has a form
$$\psi_1(t)=\left\{\begin{array}{lr}
A_1e^{\lambda_1t}+A_2e^{\lambda_2t}-C_{23}^2/(C_{23}^1u_2),\,\,
\lambda_{1,2}=u_2\left(C_{23}^3\pm\sqrt{B}\right)/2,\quad\text{if } B>0, \\
(A_1t+A_2)e^{\frac{1}{2}C_{23}^3u_2 t}-C_{23}^2/(C_{23}^1u_2),\quad\text{if }B=0, \\
e^{\frac{1}{2}C_{23}^3u_2t}\left(A_1\cos\frac{u_2\sqrt{-B}t}{2}+A_2\sin\frac{u_2\sqrt{-B}t}{2}\right)-C_{23}^2/(C_{23}^1u_2),\quad\text{if }B<0,
\end{array}\right.$$
where $A_1,\,A_2$ are arbitrary reals. Since the function $\psi_1(t)$, $t\in\mathbb{R}$, is bounded, 
this implies that either $\psi_1(t)=-C_{23}^2/(C_{23}^1u_2)$ or
\begin{equation}
\label{equat0}
\psi_1(t)=A_1\cos\left(u_2\sqrt{C_{23}^1}t\right)+A_2\sin\left(u_2\sqrt{C_{23}^1}t\right)-\frac{C_{23}^2}{C_{23}^1u_2},\,\,\text{if } C_{23}^1>0,\,\,C_{23}^3=0.
\end{equation}
	
Set $\psi_1(t)=-C_{23}^2/(C_{23}^1u_2)$. Then $\psi_3(t)=0$ on the ground of the first equation in (\ref{system}) and, taking into account  (\ref{u13}), the second equality in (\ref{fp}) is written as 
\begin{equation}
\label{eee}
F_U(-C_{23}^2s/C_{23}^1,s,0)=1/F(0,s,0).
\end{equation}
	
Let the function $\psi_1(t)$, $t\in\mathbb{R}$ be given by (\ref{equat0}). 
Notice that its range is a segment $$\left[-\sqrt{A_1^2+A_2^2}-C_{23}^2/(C_{23}^1u_2),\sqrt{A_1^2+A_2^2}-C_{23}^2/(C_{23}^1u_2)\right].$$
The first equation in (\ref{system}) and (\ref{equat0}) imply that
$$\psi_3(t)=A_1\sqrt{C_{23}^1}\sin\left(u_2\sqrt{C_{23}^1}t\right)-A_2\sqrt{C_{23}^1}\cos\left(u_2\sqrt{C_{23}^1}t\right),$$
and the range of the function $\psi_3(t)$, $t\in\mathbb{R}$, is a segment
$$\left[-\sqrt{C_{23}^1\left(A_1^2+A_2^2\right)},\sqrt{C_{23}^1\left(A_1^2+A_2^2\right)}\right].$$
This, the range of the function  $\psi_1(t)$, convexity and boundedness of the body $U^{\ast}$, the second equality in (\ref{fp}) again imply (\ref{eee}).

{\sl Sufficiency.} Assume that $C_{23}^1=C_{23}^2=0$ and $F_U(k(s),s,0)=1/F(0,s,0)$ for some $k(s)\in\mathbb{R}$.  We put
$$\psi_1(t)=k(s)F(0,s,0),\quad \psi_2(t)=sF(0,s,0),\quad\psi_3(t)=0,\quad\psi_4(t)=0.$$
It follows from Proposition \ref{const} and (\ref{u13}) that the functions $\psi_i(t)$, $i=1,\dots,4$, satisfy (\ref{system}) and (\ref{fp}) is valid.  Then abnormal extremal (\ref{dif}) satisfies the PMP with $M(t)\equiv 1$ (see (\ref{m})), and so (\ref{dif}) is nonstrongly abnormal.
	
Let $C_{23}^1\neq 0$ and (\ref{eee}) be true. We put
$$\psi_1(t)=-sC_{23}^2F(0,s,0)/C_{23}^1,\quad \psi_2(t)=sF(0,s,0),\quad \psi_3(t)=0,$$
$\psi_4(t)$ is the particular solution of the first order differential equation
$$\psi_4'-u_2C_{24}^4\psi_4+C_{24}^1C_{23}^2/C_{23}^1=0.$$
Taking into account (\ref{dif}), it is easy to see that $\psi_i(t)$, $i=1,\dots,4$, satisfy (\ref{system}) and (\ref{fp}) is hold. Then abnormal extremal (\ref{dif}) satisfies the PMP with $M(t)\equiv 1$ (see (\ref{m})), and hence (\ref{dif}) is nonstrongly abnormal.
Theorem \ref{main1} is proved.
\end{proof}		

\begin{theorem}
\label{main2}
Abnormal extremal (\ref{dif}) of a $n$-dimensional connected Lie group  $G$, $n\geq 4$, with a left-invariant sub-Riemannian
metric $d$ defined by a scalar product on  a $(n-1)$-dimensional subspace $\mathfrak{q}$ of the Lie algebra $\mathfrak{g}$ of the Lie group $G$, generating $\mathfrak{g}$,  is a geodesic, i.e. its sufficiently small segments are the shortest curves, if and only if the extremal is nonstrongly abnormal.
\end{theorem}

\begin{proof}
Sufficiency is a consequence of the statement that every normal sub-Riemannian extremal is a geodesic, which was proved in Appendix C in the memoir by W.~Liu and H.~Sussmann \cite{LS}. Necessity: a strongly abnormal sub-Riemannian extremal is not a geodesic due to the equality 
$\mathfrak{q}+[\mathfrak{q},\mathfrak{q}]=\mathfrak{q}$ and the Goch condition for abnormal geodesics (\cite{Sach}, p.~20.5.1).
\end{proof}

\section{Algebraic results using the classification}

Further, usually without mention, we use Table \ref{Tab:list}.

We get from Proposition \ref{gen1} 

\begin{corollary}
\label{nonex}
All four-dimensional real Lie algebras but $4\mathfrak{g}_{1}$ and $\mathfrak{g}^{1,1}_{4,5}$, have three-dimensional generating subspaces. 
\end{corollary}

Further, $\langle\cdot\rangle$ denotes the linear span of the vectors indicated in the parentheses.

\begin{table}[h]
\centering
\begin{tabular}{|c|c|c|}
\hline
\centering{Type of a Lie algebra}&$k$&Nonzero commutators\\
\hline
\centering{$4\mathfrak{g}_{1}$}&$0$&$-$\\
\hline
\centering{$\mathfrak{g}_{2,1}\oplus 2\mathfrak{g}_1$}&$1$&$[E_1,E_2] = E_1$\\
\hline
\centering{$2\mathfrak{g}_{2,1}$}&$2$&$[E_1,E_2] = E_1,\quad [E_3,E_4]=E_3$\\
\hline
\centering{$\mathfrak{g}_{3,1}\oplus \mathfrak{g}_1$}&$1$&$[E_2,E_3] = E_1$\\
\hline
\centering{$\mathfrak{g}_{3,2}\oplus\mathfrak{g}_1$}&$3$&$ [E_2,E_3] = E_1-E_2,\quad [E_3,E_1] = E_1$\\
\hline
\centering{$\mathfrak{g}_{3,3}\oplus\mathfrak{g}_1$}&$1$&$[E_2,E_3] = -E_2,\quad [E_3,E_1] = E_1$\\
\hline
\centering{$\mathfrak{g}^{\alpha}_{3,4}\oplus\mathfrak{g}_1$,\quad\tiny{$0\leq\alpha\neq 1$}}&$3,\,4$&$[E_2,E_3]=E_1 -\alpha E_2,\quad [E_3,E_1] = \alpha E_1- E_2$\\
\hline
\centering{$\mathfrak{g}^{\alpha}_{3,5}\oplus\mathfrak{g}_1$,\quad\tiny{$\alpha\geq 0$}}&$2$&$[E_2,E_3] = E_1-\alpha E_2,\quad [E_3,E_1]=\alpha E_1+E_2$\\
\hline
\centering{$\mathfrak{g}_{3,6}\oplus\mathfrak{g}_1$}&$5$&$[E_2,E_3] = E_1,\quad [E_3,E_1] = E_2,\quad [E_1,E_2] = -E_3$\\
\hline
\centering{$\mathfrak{g}_{3,7}\oplus\mathfrak{g}_1$}&$2$&$[E_2,E_3] = E_1,\quad [E_3,E_1] = E_2,\quad [E_1,E_2] = E_3$\\	
\hline
\centering{$\mathfrak{g}_{4,1}$}&$2$&$[E_2,E_4] = E_1,\quad [E_3,E_4] = E_2$\\
\hline
\centering{$\mathfrak{g}^{\alpha}_{4,2}$,\quad\tiny{$\alpha\neq 0$}}&$1,\,3$&$[E_1,E_4] =\alpha E_1,\,\,[E_2,E_4] = E_2,\,\, [E_3,E_4]=E_2+E_3$\\
\hline
\centering{$\mathfrak{g}_{4,3}$}&$3$&$[E_1,E_4] = E_1,\quad [E_3,E_4] = E_2$\\
\hline
\centering{$\mathfrak{g}_{4,4}$}&$2$&$[E_1,E_4] = E_1,\quad [E_2,E_4] = E_1+ E_2,$\\
\centering{$ $}&$ $&$[E_3,E_4]=E_2+E_3$\\
\hline
\centering{$\mathfrak{g}^{\alpha,\beta}_{4,5}$}&$0,1,4$&$[E_1,E_4] = E_1,\quad [E_2,E_4] = \beta E_2,\quad [E_3,E_4]=\alpha E_3$\\
\centering{\tiny{$-1<\alpha\leq \beta\leq 1,\,\,\alpha\beta\neq 0$}}&$ $&$ $\\
\centering{\tiny{or $\alpha=-1,\,\,0<\beta\leq 1$}}&$ $&$ $\\
\hline
\centering{$\mathfrak{g}^{\alpha,\beta}_{4,6}$,\quad\tiny{$\alpha>0,\,\,\beta\in\mathbb{R}$}}&$2$&$[E_1,E_4] = \alpha E_1,\quad [E_2,E_4] = \beta E_2-E_3,$\\ \centering{$ $}&$ $&$[E_3,E_4]=E_2+\beta E_3$\\
\hline
\centering{$\mathfrak{g}_{4,7}$}&$2$&$[E_1,E_4] = 2E_1,\quad [E_2,E_4] = E_2,$\\
\centering{$ $}&$ $&$[E_3,E_4]=E_2+E_3,\quad [E_2,E_3]=E_1$\\
\hline
\centering{$\mathfrak{g}^{-1}_{4,8}$}&$2$&$[E_2,E_3] = E_1,\quad [E_2,E_4] = E_2,\quad [E_3,E_4]=-E_3$\\
\hline
\centering{$\mathfrak{g}^{\alpha}_{4,8}$,\quad\tiny{$-1<\alpha\leq 1$}}&$1,\,2$&$[E_1,E_4] =(1+\alpha)E_1,\quad [E_2,E_4]=E_2,$\\
\centering{$ $}&$ $&$[E_3,E_4]=\alpha E_3,\quad [E_2,E_3]=E_1$\\
\hline
\centering{$\mathfrak{g}^{\alpha}_{4,9}$,\quad\tiny{$\alpha\geq 0$}}&$2$&$[E_1,E_4] =2\alpha E_1,\quad [E_2,E_4]=\alpha E_2-E_3,$\\
\centering{$ $}&$ $&$[E_3,E_4]=E_2+\alpha E_3,\quad [E_2,E_3]=E_1$\\
\hline
\centering{$\mathfrak{g}_{4,10}$}&$1$&$[E_1,E_3] =E_1,\quad [E_2,E_3]=E_2,$\\
\centering{$ $}&$ $&$[E_1,E_4]=-E_2,\quad [E_2,E_4]=E_1$\\
\hline
\end{tabular}
\Table\label{Tab:list} Four-dimensional real Lie algebras $\mathfrak{g}$, $k$ is a number of equivalence classes of three-dimensional subspaces generating the Lie algebra $\mathfrak{g}$
\end{table}

\begin{proposition}
\label{nonstr}
Let $\mathfrak{q}$ be a three-dimensional generating subspace of a four-dimensional Lie algebra $\mathfrak{g}$ such that $\{0\}\neq[\mathfrak{N}(\mathfrak{q}),\mathfrak{q}]\not\supset\mathfrak{N}
(\mathfrak{q}).$ Then $\mathfrak{g}=\mathfrak{g}^1_{4,8}$ or there exists a basis $(e_1,e_2,e_3:=[e_1,e_2])$ in $\mathfrak{q}$
as in Lemma \ref{base}.
\end{proposition}

\begin{proof}
Assume that $\mathfrak{N}(\mathfrak{q})=\langle e_2 \rangle$ and  $\mathfrak{s}:=[e_2,\mathfrak{q}]$. Then $\mathfrak{s}\cap 
\langle e_2 \rangle=\{0\}$ and either
1) $\dim(\mathfrak{s})=1$ or 2) $\dim(\mathfrak{s})=2$.

1) It's clear that the operator $\rm{ad}(e_2): \mathfrak{q}\rightarrow \mathfrak{q}$ has the eigenvalue $0$ of multiplicity 2 and the real eigenvalue $\alpha\neq 0$ of multiplicity 1. There exist corresponding eigenvectors
$e,f\in \mathfrak{q}$ such that 
$e\nparallel e_2$, $f\nparallel e_2$, $[e_2,e]=0,$ $[e_2,f]=\alpha f$. Setting $e_1=-(e+f),$ $e_3=\alpha f,$ we see that
$e_3=[e_1,e_2]$, $(e_1,e_2,e_3)$ is a basis in $\mathfrak{q}$ from Lemma \ref{base}.

2) The operator 
$\rm{ad}(e_2): \mathfrak{s}\rightarrow \mathfrak{s}$ has 
two (possibly coinciding) nonzero eigenvalues.
The following cases are possible: a)  the eigenvalues are conjugate
and purely imaginary; b) the eigenvalues are real and equal to $\alpha\neq\beta$;
c) the eigenvalues are real and equal to $\alpha=\beta$.

a) Let $e_1$ be an arbitrary nonzero vector from $\mathfrak{s}.$
Then the vectors $e_1,$ $e_2,$ $e_3:=[e_1,e_2]$ constitute a basis in $\mathfrak{q}$
from Lemma 1.

b) Let $e$ and $f$ be nonzero eigenvectors in $\mathfrak{s}$ with eigenvalues
$\alpha$ and $\beta$ respectively. Setting
$e_1=-(e+f),$ we get that vectors  $e_1,$ $e_2,$ $e_3:=[e_1,e_2]=\alpha e+\beta f$ constitute a basis in $\mathfrak{q}$ from Lemma \ref{base}.

c) Multiplying, if it is necessary, the vector $e_2$ by $1/\alpha$, we can assume that
$\alpha=1.$ Two subcases are possible here: there exists a basis $(e,f)$ in $\mathfrak{s}$ such that $[e_2,e]=e+f,$ $[e_2,f]=f$; 
or $[e_2,e]=e$ for any $e\in\mathfrak{s}.$ In the first subcase, we set $e_1=f-e,$ $e_3=e$ and get a basis  $(e_1,e_2,e_3=[e_1,e_2])$
from Lemma \ref{base}. In the second subcase, let $e_1,$ $e_3$ be an arbitrary linearly independent vectors from $\mathfrak{s}.$
Then $[e_2,e_1]=e_1,$ $[e_2,e_3]=e_3,$ $e_4:=[e_1,e_3]\not\in\mathfrak{q}$ and $(e_1,e_2,e_3,e_4)$ is a basis in
$\mathfrak{g}$ from Lemma \ref{base}. Further, by virtue of the Jacobi identity for vectors $e_1,e_2,e_3$, we have
$$[e_2,e_4]=[e_2,[e_1,e_3]]=[[e_2,e_1],e_3]+[e_1,[e_2,e_3]]=[e_1,e_3] + [e_1,e_3]=2e_4.$$
Thus, $e_1,e_2,e_3,e_4$ are eigenvectors of the operator $\rm{ad}(e_2):\mathfrak{g}\rightarrow \mathfrak{g}$
with eigenvalues $1$, $0$, $1$, $2$ respectively. Moreover, by the Jacobi identity for vectors $e_1,e_2,e_4$, we have
$$[e_2,[e_1,e_4]]=[[e_2,e_1],e_4]+[e_1,[e_2,e_4]]=[e_1,e_4]+2[e_1,e_4]=3[e_1,e_4]$$
and similarly $[e_2,[e_3,e_4]]=3[e_3,e_4].$ Due to what was said earlier, the number $3$ is not an eigenvalue of the operator
$\rm{ad}(e_2):\mathfrak{g}\rightarrow \mathfrak{g}$. Therefore $[e_1,e_4]=[e_3,e_4]=0.$
Setting $e_1=E_2$, $e_2=-E_4,$ $e_3=E_3,$ $e_4=E_1,$ we see from Table \ref{Tab:list}
that $\mathfrak{g}=\mathfrak{g}^1_{4,8}.$ 
\end{proof}

\begin{remark}
\label{rem1}
All four-dimensional real Lie algebras $\mathfrak{g}$ but $4\mathfrak{g}_1$, 
$\mathfrak{g}_{2,1}\oplus 2\mathfrak{g}_{1}$, $2\mathfrak{g}_{2,1}$, $\mathfrak{g}_{3,1}\oplus \mathfrak{g}_{1}$, $\mathfrak{g}_{3,3}\oplus \mathfrak{g}_{1}$, $\mathfrak{g}_{4,2}^1$, $\mathfrak{g}_{4,5}^{\alpha,1}$, $-1\leq\alpha\leq 1$, $\alpha\neq 0$, $\mathfrak{g}_{4,5}^{\alpha,\alpha}$, $-1<\alpha<1$, $\alpha\neq 0$, $\mathfrak{g}_{4,8}^1$, have two-dimensional generating subspaces \cite{BerZub21}. 
\end{remark}

\begin{proposition}
\label{ideal}
Let  $\mathfrak{g}$ be a four-dimensional real Lie algebra, for which there exists a generating two-dimensional subspace. A three-dimensional subspace $\mathfrak{q}$ generates the Lie algebra $\mathfrak{g}$ and contains no one-dimensional ideal of this algebra if and only if
there exists a basis $(e_1,e_2,e_3:=[e_1,e_2])$ in $\mathfrak{q}$ satisfying Lemma \ref{base}.
\end{proposition}

\begin{proof}
The sufficiency is obvious. Let us prove the necessity.	

In consequence of Proposition \ref{nonstr} and Remark \ref{rem1}, it remains to consider the case when
$e_2\in \mathfrak{s}\neq \langle e_2\rangle$. It is clear that $\rm{ad}(e_2)$ maps isomorphically any two-dimensional subspace $\mathfrak{p}\subset\mathfrak{q}$  that does not contain $e_2,$ on  $\mathfrak{s}.$ Since $\dim(\mathfrak{p}\cap\mathfrak{s})=1$ then there exists a vector $e_1\in \mathfrak{p}$ such that $e_1\not\in\mathfrak{s}$ and 
$[e_1,e_2]\not\in\langle e_2\rangle.$ Hence the vectors $e_1,$ $e_2,$ $e_3:=[e_1,e_2]$ constitute a basis $\mathfrak{q}$ satisfying Lemma \ref{base}. 
\end{proof}

Subspaces  $\mathfrak{q}_1,\mathfrak{q}_2\subset\mathfrak{g}$  are said to be equivalent if $\mathfrak{q}_2=\xi(\mathfrak{q}_1)$ for some  automorphism $\xi$ of the Lie algebra $\mathfrak{g}$.

\begin{proposition}
\label{eq48}
Any two generating the Lie algebra $\mathfrak{g}^1_{4,8}$ three-dimensional subspaces are equivalent.
\end{proposition}

\begin{proof}
It is easy to see that $\mathfrak{g}_{4,8}^{1}$ has only one one-dimensional ideal $\langle E_1\rangle$ and each two-dimensional subspace $\mathfrak{J}\subset\mathfrak{g}^{\prime}$ containing $E_1$ is an ideal of this Lie algebra. Then, by Proposition \ref{por},
every three-dimensional generating subspace $\mathfrak{q}$ of $\mathfrak{g}_{4,8}^{1}$ does not contain $E_1$. Hence  Proposition \ref{eq48} follows from Remark \ref{rem1}, Propositions \ref{nonstr}, \ref{ideal}, and their proofs.	
\end{proof}

\begin{proposition}
\label{nonstr1}	
Abnormal extremal (\ref{dif}) (and every its left shift) of a connected Lie group $G$ with Lie algebra $\mathfrak{g}^1_{4,8}$ and
left-invariant sub-Finsler quasimetric  $d$ defined by a seminorm $F$ on the three-dimensional generating subspace $\mathfrak{q}\subset\mathfrak{g}^1_{4,8}$ with a basis $(e_1,e_2,e_3)$ such that $[e_2,e_k]=e_k$, $k=1,3,$ is nonstrongly abnormal if and only if $F_U(0,s,0)=1/F(0,s,0)$. 
\end{proposition}

\begin{proof}
It follows from Proposition \ref{eq48} and the proof of Proposition \ref{nonstr} that for any three-dimensional
subspace $\mathfrak{q}$, generating the Lie algebra  $\mathfrak{g}^1_{4,8}$, there exists a basis $(e_1,e_2,e_3)$ such that $[e_1,e_2]=-e_1$, $[e_2,e_3]=e_3$.
Taking into account the equalities (\ref{u13}), equations (\ref{psi1})--(\ref{psi4}) can be rewritten as a system
$$\psi_1^{\prime}=u_2\psi_1,\,\,\psi_2^{\prime}=0,\,\,\psi_3^{\prime}=u_2\psi_3,\,\, \psi_4^{\prime}=2u_2\psi_4,$$
whose general solution has a form
$$\psi_1(t)=\varphi_1e^{u_2t},\,\,\psi_2(t)=\varphi_2,\,\,\psi_3(t)=\varphi_3e^{u_2t},\,\,\psi_4(t)=\varphi_4e^{2u_2t},$$
where $\varphi_i$, $i=1,\dots,4$, are some constant numbers. 
	
According to PMP and (\ref{u13}), abnormal extremal (\ref{dif}) (and every its left shift) in the space $(G,d)$ is nonstrongly abnormal if and only if  there are $\varphi_1,\varphi_3\in\mathbb{R}$ such that $F_U\left(\varphi_1e^{u_2t},sF(0,s,0),\varphi_3e^{u_2t}\right)=1$  for almost all $t\in\mathbb{R}$. This and the
boundedness condition for the body $U^{\ast}$, dual to $U$, imply that $\varphi_1=\varphi_3=0$, i.e. $F_U(0,s,0)=1/F(0,s,0)$.	
\end{proof}	

Further, $(\cdot,\cdot)$ denotes the scalar product with orthonormal
basis $(e_1,e_2,e_3)$ from Proposition \ref{nonstr1} or Theorem \ref{main1}.

\begin{theorem}
Let $d$ be a left-invariant sub-Riemannian metric on a connected four-dimensional Lie group $G$ with a Lie algebra $(\mathfrak{g},[\cdot,\cdot])$ defined by a scalar product 
$(\cdot,\cdot)_1$ on three-dimensional subspace $\mathfrak{q}\subset\mathfrak{g}$ with orthonormal
basis $(\tilde{e}_1,\tilde{e}_2,\tilde{e}_3),$ where $\tilde{e}_2\parallel e_2,$ satisfying the conditions of Lemma \ref{base}. Аbnormal extremal (\ref{dif}) is nonstrongly abnormal if and only if $\mathfrak{s}=[e_2,\mathfrak{q}]\subset \langle \tilde{e}_1, \tilde{e}_3\rangle$.
\end{theorem}

\begin{proof}
According to Propositions \ref{threed}, \ref{nonstr}, we consider five cases: two cases from Proposition \ref{threed}, 
one case in Proposition \ref{nonstr1} and two cases in Theorem \ref{main1}. 
	
The criterion $\langle e_2\rangle =\mathfrak{C}(\mathfrak{q})$ for nonstrong abnormality of the extremal  (\ref{dif}) for any quasimetric  (in particular, the sub-Riemannian metric) $d$ on $G$ is equivalent to the equality $\mathfrak{s}=\{0\}$, i.e. 
$\mathfrak{s}\subset  \langle \tilde{e}_1, \tilde{e}_3\rangle.$  
	
The extremal (\ref{dif}) is strongly abnormal for any quasimetric (in particular, the sub-Riemannian metric) $d$ on $G$ for $\langle e_2\rangle \subset \mathfrak{s},$ which always gives $\mathfrak{s}\not\subset \langle \tilde{e}_1, \tilde{e}_3\rangle.$
	
Set $\mathfrak{g}=\mathfrak{g}^1_{4,8}.$ By the last equality in Proposition
\ref{nonstr1} for $F(u)=\sqrt{(u,u)_1}$, $se_2$ is $(\cdot,\cdot)$-orthogonal
to the tangent plane for $\partial U$ at the point
$se_2/F(e_2).$ This plane is parallel to $\langle \tilde{e}_1, \tilde{e}_3\rangle,$ and  $\langle \tilde{e}_1, \tilde{e}_3\rangle= \langle e_1, e_3\rangle =\mathfrak{s}.$ 
	
In the first case of Theorem \ref{main1}, $\mathfrak{s}=\langle e_3\rangle,$ and it follows from the corresponding
equality for $F_{U}$ in the case $F(u)=\sqrt{(u,u)_1}$ that
$(k(s),s,0)$ is $(\cdot,\cdot)$-orthogonal to the tangent plane to $\partial U$ at the point
$se_2/F(e_2).$ This plane is parallel to $\langle \tilde{e}_1, \tilde{e}_3\rangle,$ and $\mathfrak{s}=\langle e_3\rangle\subset \langle \tilde{e}_1, \tilde{e}_3\rangle.$ 
	
In the second case of Theorem \ref{main1} we get that for $F(u)=\sqrt{(u,u)_1},$ the vector $(-C^2_{23}s/C^1_{23},s,0)$ and the tangent plane to $\partial U$ 
at the point $se_2/F(e_2)$ are orthogonal relative to $(\cdot,\cdot).$  This plane is parallel to the plane $\langle \tilde{e}_1, \tilde{e}_3\rangle,$ which contains the vectors $[e_2,e_3]=C^1_{23}e_1+ C^2_{23}e_2,$ $[e_1,e_2]=e_3,$ therefore $\mathfrak{s}=\langle \tilde{e}_1, \tilde{e}_3\rangle.$ 
\end{proof} 

\begin{proposition}
\label{afterrem}
If a four-dimensional real Lie algebra $\mathfrak{g}$ has a two-dimensional generating subspace $\mathfrak{p}_0$
then there exists a basis $(e_1,e_2,e_3,e_4)$ for $\mathfrak{g}$ such that $e_1,e_2\in\mathfrak{p}_0$ and $[e_1,e_2]=e_3$, $[e_1,e_3]=e_4$, $C_{23}^4=0.$ 
A three-dimensional subspace $\mathfrak{q}_0:=\mathfrak{p}_0\oplus [\mathfrak{p}_0,\mathfrak{p}_0]$
generates the Lie algebra $\mathfrak{g}$ and the basis $(e_1,e_2,e_3)$ of the subspace $\mathfrak{q}_0$ satisfies Lemma \ref{base}. If  
$\tilde{\mathfrak{p}}\subset\mathfrak{g}$ is equivalent to $\mathfrak{p}_0$ then the corresponding three-dimensional subspaces 
$\tilde{\mathfrak{q}}=\tilde{\mathfrak{p}}\oplus [\tilde{\mathfrak{p}},\tilde{\mathfrak{p}}]$ and $\mathfrak{q}_0$ are equivalent.
\end{proposition}

\begin{proof}
The first statement of Proposition	\ref{afterrem} was proved in \cite{BerZub21}. It immediately implies the second statement of this proposition. 
If $\tilde{\mathfrak{p}}=\xi(\mathfrak{p}_0)$ for some automorphism $\xi$ of the Lie algebra $\mathfrak{g}$ then $(\xi(e_1),\xi(e_2),\xi(e_3))$ is a basis of the subspace $\tilde{\mathfrak{q}}$ because $\xi(e_3)=\xi([e_1,e_2])=[\xi(e_1),\xi(e_2)]$. Consequently, $\tilde{\mathfrak{q}}=\xi(\mathfrak{q}_0)$.
\end{proof}

\begin{proposition}
\label{solv2}
A three-dimensional subspace $\mathfrak{q}\subset \mathfrak{g}$ generates $\mathfrak{g}=2\mathfrak{g}_{2,1}$ if and only if ${\rm dim}(\mathfrak{q}\cap\mathfrak{J})=1$ for every two-dimensional ideal $\mathfrak{J}$ of the Lie algebra $\mathfrak{g}$. 
There exist two equivalence classes of such subspaces; 
$\mathfrak{q}$ belongs to the first (the second) equivalence class if $\mathfrak{q}$ contains one (contains no) one-dimensional ideal of  $\mathfrak{g}$. 
\end{proposition}

\begin{proof}
Let $\mathfrak{g}^1_{2,1}$ and $\mathfrak{g}^2_{2,1}$ denote the first and the second copies of the Lie algebra $\mathfrak{g}_{2,1}$ respectively, and
$\mathfrak{L}_1,\mathfrak{L}_2$ be their one-dimensional ideals. The Lie algebra $\mathfrak{g}$ has three two-dimensional ideals:
$\mathfrak{g}^1_{2,1},$ $\mathfrak{g}^2_{2,1},$ $\mathfrak{g}'=\mathfrak{L}_1\oplus \mathfrak{L}_2.$
In consequence of Proposition \ref{por}, the intersection of a subspace $\mathfrak{q}$ with each of them is one-dimensional.
Two cases are possible:
	
(1) the intersection of a subspace $\mathfrak{q}$ with exactly one of the Lie algebras $\mathfrak{g}^1_{2,1},$  
$\mathfrak{g}^2_{2,1}$ is a one-dimensional ideal; without loss of generality,
we can assume that $\mathfrak{q}\cap \mathfrak{g}^2_{2,1}=\mathfrak{L}_2;$ 
	
(2) $\mathfrak{q}\cap \mathfrak{g}^k_{2,1}\neq \mathfrak{L}_k,$ $k=1,2.$ 
	
It is clear that subspaces $\mathfrak{q}$ of the first and the second types are
not equivalent.
	
In the first case, there exists a basis $(e_1,e_2,e_3)$ for  $\mathfrak{q}$ such that
$$e_1\in \mathfrak{q}\cap \mathfrak{g}^1_{2,1},\quad e_2\in \mathfrak{L}_2,\quad
e_3=f_1 + f_2,\quad 0\neq f_1\in \mathfrak{L}_1,\quad f_2\in \mathfrak{g}^2_{2,1},\quad f_2\notin\mathfrak{L}_2,$$
moreover, $[e_1,e_3]=f_1=e_4$, $[e_2,e_3]=e_2$, $[e_1,e_4]=e_4,$
and all other Lie brackets for the basis vectors  $e_1,e_2,e_3,e_4$ are zero.
It follows from this that any two subspaces  $\mathfrak{q}$ of the first type are equivalent.

In the second case, there exists a basis $(e_1,e_2,e_3)$ for  $\mathfrak{q}$ such that
$$e_1\in \mathfrak{q}\cap \mathfrak{g}^1_{2,1},\quad e_2=e_1+f_2,\quad f_2\in \mathfrak{q}\cap\mathfrak{g}^2_{2,1},\quad
e_3\in\mathfrak{g}',\quad e_3\notin\mathfrak{L}_1,\quad e_3\notin\mathfrak{L}_2,$$
and the components of the vector $e_2$ at $E_2$ and $E_4$ are equal to $1$. 
Then $e_4:=[e_1,e_3]\in\mathfrak{L}_1$, $[e_2,e_3]=-e_3$, $[e_1,e_4]=[e_2,e_4]=-e_4,$
and all other Lie brackets for the basis vectors $e_1,e_2,e_3,e_4$ are zero.
It follows from this that any two subspaces  $\mathfrak{q}$ of the second type are equivalent.	
\end{proof}

\begin{proposition}
\label{solv3,6}
A three-dimensional subspace $\mathfrak{q}$ of the Lie algebra $\mathfrak{g}=\mathfrak{g}_{3,6}\oplus \mathfrak{g}_1$
generates $\mathfrak{g}$ if and only if $\mathfrak{q}\neq\mathfrak{g}_{3,6}$ and the projection of $\mathfrak{q}$ to 
$\mathfrak{g}_{3,6}$ along $\mathfrak{g}_1$ isn't a two-dimensional subalgebra of the Lie algebra $\mathfrak{g}_{3,6}$.
There exist five equivalence classes of such subspaces.
\end{proposition}

\begin{proof}
The first statement follows from Proposition \ref{gen}. According to \cite{Patera}, all two-dimensional Lie subalgebras of the Lie algebra $\mathfrak{g}_{3,6}$ are equivalent to $\langle E_1-E_3, E_2 \rangle$.

Let $\mathfrak{g}_1\subset\mathfrak{q}$. By Lemma \ref{base}, $\mathfrak{q}_1=\mathfrak{g}_1$ and since $\mathfrak{p}\cap\mathfrak{q}_1=\{0\}$ for
the two-dimensional space $\mathfrak{p}:=\mathfrak{g}_3\cap\mathfrak{q}$ then 
$\mathfrak{p}$ generates $\mathfrak{g}_3$. In \cite{AB} was proved that there exist two equivalence classes of such spaces
$\mathfrak{p}$ so the similar statement holds for three-dimensional subspaces $\mathfrak{q}$,
generating the Lie algebra $\mathfrak{g}$ and containing $\mathfrak{g}_1$.

Now assume that $\mathfrak{g}_1\not\subset\mathfrak{q}$, i.e. $\mathfrak{q}$  contains no one-dimensional ideal of $\mathfrak{g}$.
It follows from Remark \ref{rem1} and Proposition \ref{ideal} that there exists a basis $(e_1,e_2,e_3:=[e_1,e_2])$ in $\mathfrak{q}$ satisfying Lemma \ref{base}, in particular, a two-dimensional subspace $\mathfrak{p}_0$ with the basis $(e_1,e_2)$ generates the Lie algebra $\mathfrak{g}$. 
In \cite{BerZub21} was proved that there exist four equivalence classes of such spaces $\mathfrak{p}_0$,
and subspaces $\langle E_1,E_2+E_4\rangle$, $\langle E_3,E_1+E_4\rangle$, $\langle E_1,E_3+E_4\rangle$, 
$\langle E_4+(E_2-E_3)/2,E_2+E_3\rangle$ belong to the first, the second, the third, the fourth equivalence classes, respectively.
It's not hard to compute, taking into account Proposition \ref{afterrem}, that for $\mathfrak{p}_{0,i}$, $i=1,\dots,4$, of $i$--th equivalence class, the three-dimensional subspace $\mathfrak{q}_{0,i}=\mathfrak{p}_{0,i}\oplus [\mathfrak{p}_{0,i},\mathfrak{p}_{0,i}]$, $i=1,\dots,4$,
has a basis  $(e_1,e_2,e_3)$ from Lemma \ref{base} such that $[e_1,e_2]=e_3$, and for $e_4=[e_1,e_3]$,
$$[e_2,e_3]=-e_1,\quad [e_2,e_4]=0,\quad [e_1,e_4]=e_3,\quad [e_3,e_4]=e_1;$$
$$[e_2,e_3]=-e_1,\quad [e_2,e_4]=0,\quad [e_1,e_4]=-e_3,\quad [e_3,e_4]=-e_1;$$
$$[e_2,e_3]=e_1,\quad [e_2,e_4]=0,\quad [e_1,e_4]=e_3,\quad [e_3,e_4]=-e_1;$$
$$[e_2,e_3]=e_2,\quad [e_2,e_4]=e_3,\quad [e_1,e_4]=0,\quad [e_3,e_4]=e_4$$
(in ascending order of $i$). With mutual permutation of vectors $e_1$ and $e_3$
the equality $[e_1,e_2]=e_3$ and the last three (respectively, the first of) relations in their first quadruple will go into the second quadruple (respectively, to $[e_1,e_2]=e_3$). Therefore, the subspaces $\mathfrak{q}_{0,1}$ and $\mathfrak{q}_{0,2}$ are equivalent.

Let us show that the subspaces $\mathfrak{q}_{0,k}$, $k=2,3,4$, are pairwise not equivalent.

Let $\xi$ be an automorphism of the Lie algebra $\mathfrak{g}_{3,6}\oplus \mathfrak{g}_1$ such that
$\xi(\mathfrak{q}_{0,k})=\mathfrak{q}_{0,j},$ where $2\leq k<j\leq 4.$
Since $\langle e_2\rangle$ are the only one-dimensional normalizers of $\mathfrak{q}_{0,k}$ and $\mathfrak{q}_{0,j}$, then  $\xi(\langle e_2\rangle)=\langle e_2\rangle$. Therefore  $j\neq 4$ because $\langle e_2\rangle\subset [\langle e_2\rangle,\mathfrak{q}_{0,4}]$, and we have $[\langle e_2\rangle,\mathfrak{q}_{0,k}]= \langle e_1, e_3\rangle=[\langle e_2\rangle,\langle e_1, e_3\rangle]$ for $k=2,3$. Eigenvalues of the
operator $\ad(e_2)$ on $\langle e_1, e_3\rangle$ are equal to $\pm 1$ for $\mathfrak{q}_{0,2}$ and  $\pm {\bf i}$ for $\mathfrak{q}_{0,3}$.
Therefore, these spectra are not real-similarly, $\mathfrak{q}_{0,2}$ and $\mathfrak{q}_{0,3}$ are not equivalent.
\end{proof}

\begin{proposition}
\label{equiv0}
Any two three-dimensional subspaces of any four-dimensional real Lie algebra $\mathfrak{g}\neq\mathfrak{sl}(2,\mathbb{R})\oplus\mathfrak{g}_1$, generating $\mathfrak{g}$ and containing the same one-dimensional ideal of this algebra, are equivalent.
\end{proposition}

\begin{proof}
Assume that different three-dimensional subspaces $\mathfrak{q},\tilde{\mathfrak{q}}\subset\mathfrak{g}$ generates a Lie algebra $\mathfrak{g}$ and contain a one-dimensional ideal $\mathfrak{L}$ of this algebra. 
	
Suppose that $\mathfrak{L}\not\subset\mathfrak{C(g)}$, $0\neq e_2\in\mathfrak{L}$. Then there exist linearly independent vectors $e_1,e_3$
from some two-dimensional subspace $\mathfrak{p}\subset\mathfrak{q}$, where
$\mathfrak{p}\cap\mathfrak{L}=\{0\}$, for which  the first two equalities indicated below are satisfied:
\begin{equation}
\label{d}
[e_1,e_2]=0,\,\,[e_2,e_3]=e_2,\,\,[e_1,e_3]=e_4,\,\,[e_2,e_4]=0.
\end{equation}
We define $e_4$ by the third equality in (\ref{d}). It follows from Lemma \ref{base} that $e_4\notin\mathfrak{q}$, therefore $(e_1,e_2,e_3,e_4)$ is a basis for $\mathfrak{g}$. The fourth relation in (\ref{d}) follows from the Jacobi identity for vectors $e_1,e_2,e_3$ and the previous equalities:
$$[e_2,e_4]=[e_2,[e_1,e_3]]=[[e_2,e_1],e_3]+[e_1,[e_2,e_3]]=[e_1,e_2]=0.$$

Equalities (\ref{d}) and the Jacobi identities for  vectors $e_1,e_2,e_4$ and $e_2,e_3,e_4$ imply 
$$C^3_{14}e_2=[e_2,[e_1,e_4]]= [[e_2,e_1],e_4] + [e_1,[e_2,e_4]]=0\,\,\Rightarrow\,\,C^3_{14}=0;$$
$$C^3_{34}e_2=[e_2,[e_3,e_4]]= [[e_2,e_3],e_4] + [e_3,[e_2,e_4]]=[e_2,e_4]=0\,\,\Rightarrow\,\,C^3_{34}=0.$$

Writing the Jacobi identity $[e_3,[e_1,e_4]]=[e_1,[e_3,e_4]]$  for vectors $e_1,e_3,e_4$,
taking into account (\ref{d}) and the equalities $C^3_{14}=0$, $C^3_{34}=0$, we get
$$-C^1_{14}e_4-C^2_{14}e_2+C^4_{14}(C^1_{34}e_1+C^2_{34}e_2+C^4_{34}e_4)=[e_3,[e_1,e_4]]=$$
$$[e_1,[e_3,e_4]]=C^4_{34}(C^1_{14}e_1+C^2_{14}e_2+C^4_{14}e_4).$$ 
Then
\begin{equation}
\label{fca}
[e_1,e_4]=C^2_{14}e_2+C^4_{14}e_4,\quad [e_3,e_4]=C^1_{34}e_1+C^2_{34}e_2+C^4_{34}e_4,
\end{equation}
moreover,
\begin{equation}
\label{fcad}
C^1_{34}C^4_{14}=0,\quad C^4_{14}C^2_{34}=C^2_{14}(C^4_{34}+1).
\end{equation}

Set $C_{14}^4=0$. As we proved, there exists a basis $(\tilde{e}_1,\tilde{e}_2,\tilde{e}_3,\tilde{e}_4)$ of the Lie algebra
$\mathfrak{g}$ such that the vectors  $\tilde{e}_1,e_2,\tilde{e}_3$  belong to  $\tilde{\mathfrak{q}}$ and
\begin{equation}
\label{d1}
[\tilde{e}_1,e_2]=0,\,\,[e_2,\tilde{e}_3]=e_2,\,\,[\tilde{e}_1,\tilde{e}_3]=\tilde{e}_4,\,\,[e_2,\tilde{e}_4]=0.
\end{equation}
Let us denote by $a_i$, $b_i$, $i=1,\dots,4$, respectively, coordinates of the vectors $\tilde{e}_1$ and $\tilde{e}_3$ in the basis $(e_1,e_2,e_3,e_4)$.
Since $\tilde{\mathfrak{q}}\neq\mathfrak{q}$ and each vector, collinear to a basis vector, 
can be added to another basis vector, we can assume that $a_4=a_2=b_2=0$, $b_4\neq 0$. 
Then due to (\ref{d}) and (\ref{d1}) we have
$$0=[\tilde{e}_1,e_2]=a_1[e_1,e_2]-a_3[e_2,e_3]=-a_3e_2\quad\Rightarrow\quad a_3=0,$$
$$e_2=[e_2,\tilde{e}_3]=-b_1[e_1,e_2]+b_3[e_2,e_3]+b_4[e_2,e_4]=b_3e_2\quad\Rightarrow\quad b_3=1.$$
Adding, if it is necessary, the vector $(b_4-b_1)\tilde{e}_1$ to the vector $\tilde{e}_3$, we can assume that $\tilde{e}_1=e_1$ and
$\tilde{e}_3=b_4e_1+e_3+b_4e_4$. Grating (\ref{d}), (\ref{fca}), and $C_{14}^4=0$, we find consequtively 
$$\tilde{e}_4=[e_1,b_4e_1+e_3+b_4e_4]=e_4+b_4[e_1,e_4]=e_4+b_4C_{14}^2e_2,$$
$$[\tilde{e}_1,\tilde{e}_4]=[e_1,e_4+b_4C_{14}^2e_2]=[e_1,e_4]+b_4C_{14}^2[e_1,e_2]=[e_1,e_4],$$
$$[\tilde{e}_3,\tilde{e}_4]=[b_4e_1+e_3+b_4e_4,e_4+b_4C_{14}^2e_2]=b_4[e_1,e_4]+[e_3,e_4]-b_4C_{14}^2[e_2,e_3]=$$
$$b_4C_{14}^2e_2+[e_3,e_4]-b_4C_{14}^2[e_2,e_3]=[e_3,e_4],\quad [e_2,e_4^{\ast}]=[e_2,e_4+b_4C_{14}^2e_2]=[e_2,e_4]=0.$$
Consequently, a linear operator $\xi$ of the Lie algebra $\mathfrak{g}$ such that $\xi(e_2)=e_2$, $\xi(e_i)=\tilde{e}_i$,  $i=1,3,4$, is an automorphism of the Lie algebra $\mathfrak{g}$ and the equality $\xi(\mathfrak{q})=\tilde{\mathfrak{q}}$ holds, i.e. the subspaces $\mathfrak{q}$ and $\tilde{\mathfrak{q}}$ are equivalent.

Assume that $C_{14}^4\neq 0$.
Using  (\ref{d}), (\ref{fca}), (\ref{fcad}), it is easy to check  that  $\mathfrak{g}$ is isomorphic to $2\mathfrak{g}_{2,1}$;
the required isomorphism $\xi:\mathfrak{g}\rightarrow 2\mathfrak{g}_{2,1}$ is given by formulas
$$\xi(e_1)=\left(-C_{14}^2/C_{14}^4\right)E_1-C_{14}^4E_4,\quad \xi(e_2)=E_1,$$
$$\xi(e_3)=E_2+E_3-C_{34}^4E_4,\quad\xi(e_4)=\left(-C_{14}^2/C_{14}^4\right)E_1+C_{14}^4E_3.$$
Then on the ground of Proposition \ref{solv2} the subspaces $\mathfrak{q}$ and $\tilde{\mathfrak{q}}$ are equivalent.

Assume that $\mathfrak{L}\subset\mathfrak{C(g)}$. At first, consider all indecomposable four-dimensional Lie algebras containing a nontrivial (one-dimensional) central ideal: $\mathfrak{g}_{4,1}$,  $\mathfrak{g}_{4,3}$, $\mathfrak{g}_{4,8}^{-1}$,  $\mathfrak{g}_{4,9}^0$. 

Set $\mathfrak{g}=\mathfrak{g}_{4,1}$ or $\mathfrak{g}=\mathfrak{g}_{4,3}$. In consequence of Proposition \ref{por}, any three-dimensional subspace $\mathfrak{q}$, generating $\mathfrak{g}$ and containing $\mathfrak{C(g)}$, has a basis $(e_1,e_2,e_3)$, where
$e_1$ does not belong to the three-dimensional ideal $\mathfrak{I}=\langle E_1,E_2,E_3\rangle$, $e_2=E_1$ in the case $\mathfrak{g}_{4,1}$ and $e_2=E_2$ in the case $\mathfrak{g}_{4,3}$, $e_3\in\mathfrak{I}$ and $e_3$ does not belong to any two-dimensional ideal
$\mathfrak{J}\subset\mathfrak{I}$, and the components of the vector $e_1$ at $E_4$ and  of the vector $e_3$ at $E_3$ in the basis $(E_1,E_2,E_3,E_4)$ are equal to $1$. 
Then $e_4:=[e_1,e_3]\notin\mathfrak{q}$, $[e_1,e_4]=e_2$ in the case $\mathfrak{g}_{4,1}$ and $[e_1,e_4]=-e_2-e_4$ in the case
$\mathfrak{g}_{4,3}$, while all other Lie brackets for vectors $e_1,e_2,e_3,e_4$ are zero.
Consequently, any two three-dimensional subspaces, generating $\mathfrak{g}$ and containing $\mathfrak{C(g)}$, are equivalent.

Assume that $\mathfrak{g}=\mathfrak{g}_{4,8}^{-1}$ or $\mathfrak{g}=\mathfrak{g}_{4,9}^{0}$. In consequence of Proposition \ref{por}, any three-dimensional subspace  $\mathfrak{q}$, generating $\mathfrak{g}$ and containing $\mathfrak{C(g)}$,  has a basis $(e_1,e_2,e_3)$, where
$e_1\notin\mathfrak{g}^{\prime}$, 
$e_2\in\mathfrak{C(g)}$, $e_3\in\mathfrak{g}^{\prime}$ and in the case $\mathfrak{g}_{4,8}^{-1},$ the vector $e_3$ does not belong to 
two-dimensional ideals $\mathfrak{J}_1=\langle E_1,E_2\rangle$ and $\mathfrak{J}_2=\langle E_1,E_3\rangle$. 
Since one can add to a vector of the basis any vector, collinear to another vector of the basis, we can assume without loss of generality that
$e_1=a_1E_1+a_2E_2+a_3E_3+E_4$,
$$\mathfrak{g}_{4,8}^{-1}:\quad e_2=a_4a_5E_1,\,\,e_3=(a_3a_4+a_2a_5)E_1+a_4E_2+a_5E_3,\quad a_4a_5\neq 0;$$
$$\mathfrak{g}_{4,9}^{0}:\quad e_2=(a_4^2+a_5^2)E_1,\,\,e_3=(a_3a_5-a_2a_4)E_1+a_4E_2-a_5E_3,\quad a_4^2+a_5^2\neq 0.$$
Then $e_4:=[e_1,e_3]\notin\mathfrak{q}$. It is easy to check that $[e_1,e_4]=e_3$, $[e_3,e_4]=2e_2$ in the case $\mathfrak{g}_{4,8}^{-1}$ and
$[e_1,e_4]=-e_3$, $[e_3,e_4]=e_2$ in the case $\mathfrak{g}_{4,9}^{0}$, 
while all other Lie brackets for vectors $e_1,e_2,e_3,e_4$ are zero.
Consequently, any two three-dimensional subspaces, generating  $\mathfrak{g}$ and containing $\mathfrak{C(g)}$, are equivalent.

Now assume that $\mathfrak{g}=\mathfrak{g}_3\oplus\mathfrak{g}_1$, $\mathfrak{g}_3\neq\mathfrak{g}_{3,1}$, $\mathfrak{g}_3\neq\mathfrak{sl}(2,\mathbb{R})$. 
By Lemma \ref{base}, $\mathfrak{q}_1=\mathfrak{g}_1$ for any three-dimensional subspace $\mathfrak{q}$, generating $\mathfrak{g}$. Since $\mathfrak{p}\cap\mathfrak{q}_1=\{0\}$ for two-dimensional subspace $\mathfrak{p}:=\mathfrak{g}_3\cap\mathfrak{q},$ then on the ground of Lemma \ref{base},  $\mathfrak{p}$ generates $\mathfrak{g}_3$. In \cite{BerZub21} was proved that in this case $\mathfrak{g}_3\neq\mathfrak{g}_{3,3}$ and any two such  subspaces  $\mathfrak{p}$  are equivalent. Consequently,  any two such three-dimensional subspaces, generating $\mathfrak{g}$ and
containing $\mathfrak{g}_1$, are equivalent.

Assume that $\mathfrak{g}=\mathfrak{g}_{3,1}\oplus\mathfrak{g}_1$. Since ${\rm dim}\,\mathfrak{g}^{\prime}=1$,
${\rm dim}\,\mathfrak{C}(\mathfrak{g})=2$ and $\mathfrak{g}^{\prime}\subset\mathfrak{C}(\mathfrak{g}),$ then any three-dimensional generating subspace $\mathfrak{q}\subset \mathfrak{g}$ does not contain $\mathfrak{g}^{\prime}$ and, due to Lemma \ref{base},
$\mathfrak{q}_1=\mathfrak{C}(\mathfrak{g})\cap\mathfrak{q}$. By virtue of Proposition \ref{por},  the subspace $\mathfrak{q}$ has a basis
$(e_1,e_2,e_3)$, where $e_1$ ($e_3$) belongs to two-dimensional commutative ideal $\mathfrak{J}_1=\langle E_1,E_2\rangle$ (respectively, $\mathfrak{J}_2=\langle E_1,E_3\rangle$), $e_1,e_3\notin\mathfrak{g}^{\prime}$, $e_2\in\mathfrak{q}_1$. Then $e_4:=[e_1,e_3]\notin\mathfrak{q}$,
all other Lie brackets for vectors $e_1,e_2,e_3,e_4$ are zero.
Consequently, any two three-dimensional subspaces, generating  $\mathfrak{g}_{3,1}\oplus\mathfrak{g}_{1}$ are equivalent.

It remains to consider the Lie algebra  $\mathfrak{g}=\mathfrak{g}_{2,1}\oplus 2\mathfrak{g}_{1}$. Since ${\rm dim}\,\mathfrak{g}^{\prime}=1$,
$\mathfrak{C(g)}=2\mathfrak{g}_{1}$ then any three-dimensional subspace  $\mathfrak{q}$ of the Lie algebra $\mathfrak{g},$ generating it, does not contain
$\mathfrak{g}^{\prime}$ and, due to Lemma \ref{base},
$\mathfrak{q}_1=\mathfrak{C}(\mathfrak{g})\cap\mathfrak{q}$. 
Since the subalgebras $\mathfrak{g}_{2,1}$, $\mathfrak{C(g)}$ are two-dimensional and $\mathfrak{g}^{\prime}\subset\mathfrak{g}_{2,1},$ then
subspaces $\mathfrak{q}\cap\mathfrak{g}_{2,1}$ and $\mathfrak{q}\cap\mathfrak{C(g)}$ are one-dimensional. Then there exists a basis  $(e_1,e_2,e_3)$  for $\mathfrak{q}$ such that 
$$e_1\in\mathfrak{g}_{2,1},\,\,e_1\notin\mathfrak{g}^{\prime},\,\,e_2\in\mathfrak{C(g)},\,\,e_3=f_1+f_2,\,\,
0\neq f_1\in\mathfrak{g}^{\prime},\,\,0\neq f_2\in\mathfrak{C(g)},$$ 
moreover, $[e_1,e_3]=e_4=f_1$, $[e_1,e_4]=e_4$, 
all other Lie brackets for vectors  $e_1,e_2,e_3,e_4$ are zero.
Consequently, any three-dimensional subspaces of the Lie algebra 
$\mathfrak{g}_{2,1}\oplus 2\mathfrak{g}_{1},$ generating it, are equivalent.
\end{proof}	

\begin{proposition}
\label{equiv}
Assume that a four-dimensional Lie algebra $\mathfrak{g}$ has three-dimensional generating subspaces. Then 

1. If $\mathfrak{g}$ has an infinite number of two-dimensional ideals then 
any two three-dimensional subspaces, generating the Lie algebra $\mathfrak{g}$, are equivalent.

2. If $\mathfrak{g}$ has a finite number (respectively, zero) of  two-dimensional ideals and $\mathfrak{g}$ is different from the Lie algebras
$2\mathfrak{g}_{2,1}$, $\mathfrak{g}_{3,6}\oplus\mathfrak{g}_{1}$, then
$\mathfrak{g}$ has a finite number $m$, $0\leq m\leq 3$, of pairwise nonequivalent one-dimensional ideals  $\mathfrak{L}_1,\dots,\mathfrak{L}_m$. 
There exist $m+1$ equivalence classes of three-dimensional subspaces, generating the Lie algebra $\mathfrak{g}$;
$\mathfrak{q}\subset\mathfrak{g}$ belongs to the $i$--th equivalence class 
($i=1,\dots,m$), if $\mathfrak{L}_i\subset\mathfrak{q}$; $\mathfrak{q}\subset\mathfrak{g}$ belongs to the $(m+1)$--th equivalence class, if $\mathfrak{q}$ contains no one-dimensional  ideal of the Lie algebra $\mathfrak{g}$.
\end{proposition}

\begin{proof}
1. It follows from Table \ref{Tab:list} and Corollary \ref{nonex} that $\mathfrak{g}$ is one of the Lie algebras $\mathfrak{g}_{2,1}\oplus 2\mathfrak{g}_{1}$, $\mathfrak{g}_{3,1}\oplus\mathfrak{g}_{1}$, $\mathfrak{g}_{3,3}\oplus\mathfrak{g}_{1}$, $\mathfrak{g}_{4,2}^1$, $\mathfrak{g}_{4,8}^1$, $\mathfrak{g}_{4,5}^{\alpha,1}$, $-1\leq \alpha<1$, $\alpha\neq 0$, $\mathfrak{g}_{4,5}^{\alpha,\alpha}$, $-1<\alpha<1$, $\alpha\neq 0$.

Lie algebras $\mathfrak{g}_{2,1}\oplus 2\mathfrak{g}_{1}$, $\mathfrak{g}_{3,1}\oplus\mathfrak{g}_{1}$ 
were considered in the proof of Proposition \ref{equiv0}.

Set $\mathfrak{g}=\mathfrak{g}_{3,3}\oplus\mathfrak{g}_{1}$. Every two-dimensional subspace in $\mathfrak{g}_{3,3}$ is a Lie algebra.
Therefore a three-dimensional subspace $\mathfrak{q}\subset \mathfrak{g}$ generates $\mathfrak{g}$ if and only if $\mathfrak{g}_{3,3}$ is a projection
of $\mathfrak{q}$ to $\mathfrak{g}_{3,3}$ along $\mathfrak{g}_1$ and $\dim(\mathfrak{q}\cap \mathfrak{g}')=1.$ 
Then any two three-dimensional subspaces, generating the Lie algebra $\mathfrak {g}$, are equivalent.

Set $\mathfrak{g}=\mathfrak{g}_{4,2}^{1}$. It is easy to see that every two-dimensional subspace  $\mathfrak{J}\subset\mathfrak{g}^{\prime}$, containing the vector $E_2$, is an ideal of the Lie algebra $\mathfrak{g}$. Then by Proposition \ref{por} every three-dimensional subspace $\mathfrak{q}$ of $\mathfrak{g},$ generating it, does not contain $E_2$ and has a basis $(e_1,e_2,e_3)$, where $e_1\notin\mathfrak{g}^{\prime}$ and the component of the vector  $e_1$ at $E_4$ is equal to $1$, $e_2\in\langle E_1,E_2\rangle$, $e_3\in\langle E_2,E_3\rangle$. 
It is easy to check that
$$[e_1,e_2]=-e_2,\quad  e_4:=[e_1,e_3]\notin\mathfrak{q},\quad [e_1,e_4]=-e_3-2e_4,$$
all other Lie brackets for vectors $e_1,e_2,e_3,e_4$ are zero.
Consequently, any three-dimensional subspace, generating $\mathfrak{g}_{4,2}^{1},$ contains a one-dimensional ideal of this Lie algebra, and any two such subspaces are equivalent.

Set $\mathfrak{g}=\mathfrak{g}_{4,5}^{\alpha,1}$ or $\mathfrak{g}=\mathfrak{g}_{4,5}^{\alpha,\alpha}$, $-1\leq \alpha<1$, $\alpha\neq 0$. 
It is easy to see that the subspace $\mathfrak{I}:=\langle E_1,E_2\rangle$ in the case $\mathfrak{g}_{4,5}^{\alpha,1}$
($\mathfrak{I}:=\langle E_2,E_3\rangle$ in the case  $\mathfrak{g}_{4,5}^{\alpha,\alpha}$) and every two-dimensional subspace 
$\mathfrak{J}\subset\mathfrak{g}^{\prime}$, containing the vector $E_3$ 
 in the case $\mathfrak{g}_{4,5}^{\alpha,1}$ ($E_1$ in the case  $\mathfrak{g}_{4,5}^{\alpha,\alpha}$), are ideals of this Lie algebra.
Then by Proposition \ref{por} every three-dimensional subspace $\mathfrak{q}$ of the Lie algebra $\mathfrak{g},$ generating it, does not contain $E_3$ in the case $\mathfrak{g}_{4,5}^{\alpha,1}$ ($E_1$ in the case  $\mathfrak{g}_{4,5}^{\alpha,\alpha}$)  and has a basis $(e_1,e_2,e_3)$, where $e_1\notin\mathfrak{g}^{\prime}$ and the component of the vector $e_1$ at $E_4$ is equal to $1$, $e_2\in\mathfrak{I}$,  $e_3=g_1+E_3$, $g_1\in\mathfrak{I}$, $g_1\nparallel e_2$.
It is easy to check that 
$$\mathfrak{g}_{4,5}^{\alpha,1}:\quad [e_1,e_2]=-e_2,\quad e_4:=[e_1,e_3]\notin\mathfrak{q},\quad
[e_1,e_4]=-\alpha e_3-(1+\alpha)e_4,$$
$$\mathfrak{g}_{4,5}^{\alpha,\alpha}:\quad [e_1,e_2]=-\alpha e_2,\quad e_4:=[e_1,e_3]\notin\mathfrak{q},\quad 
[e_1,e_4]=-\alpha e_3-(1+\alpha)e_4,$$
all other Lie brackets for vectors $e_1,e_2,e_3,e_4$ are zero.
Consequently, any three-dimensional subspace, generating $\mathfrak{g},$  contains a one-dimensional ideal of this Lie algebra, and any two such subspaces are equivalent.

The statement for the Lie algebra $\mathfrak{g}_{4,8}^{1}$  is proved in Proposition \ref{eq48}.

2. Let $\mathfrak{g}$ be different from the Lie algebras $2\mathfrak{g}_{2,1}$, $\mathfrak{g}_{3,6}\oplus\mathfrak{g}_{1}$ and have a 
finite number (respectively, zero) of  two-dimensional ideals. It follows from Table \ref{Tab:list} that

1) $m=0$ for $\mathfrak{g}=\mathfrak{g}_{4,10}$;

2)  $m=1$ for decomposable Lie algebras $\mathfrak{g}_{3,5}^{\alpha}\oplus\mathfrak{g}_{1}$, $\alpha\geq 0$, $\mathfrak{g}_{3,7}\oplus\mathfrak{g}_{1}$ (the center is $\langle E_4\rangle$), for indecomposable Lie algebras $\mathfrak{g}_{4,1}$, $\mathfrak{g}_{4,4}$, $\mathfrak{g}_{4,6}^{\alpha,\beta}$, $\alpha>0$, $\beta\in\mathbb{R}$, $\mathfrak{g}_{4,7}$, $\mathfrak{g}_{4,8}^{\alpha}$, $-1\leq\alpha<1$, $\mathfrak{g}_{4,9}^{\alpha}$, $\alpha\geq 0$ (one-dimensional ideal is $\langle E_1\rangle$);

3) $m=2$ for the Lie algebras 
a) $\mathfrak{g}_{3,2}\oplus\mathfrak{g}_{1}$ (with one-dimensional ideals $\langle E_1\rangle$ and $\mathfrak{g}_{1}$),
b) $\mathfrak{g}_{3,4}^0\oplus\mathfrak{g}_{1}$ (with central ideal $\mathfrak{g}_{1}$ and noncentral mutually equivalent one-dimensional ideals $\langle E_1+E_2\rangle$ and $\langle E_1-E_2\rangle$ are equivalent), c) $\mathfrak{g}_{4,2}^{\alpha}$, $\alpha\notin\{0,1\}$ and
 $\mathfrak{g}_{4,3}$ (with one-dimensional ideals $\langle E_1\rangle$ and $\langle E_2\rangle$);

4) $m=3$ for the Lie algebras a) $\mathfrak{g}_{3,4}^{\alpha}\oplus\mathfrak{g}_{1}$, $0<\alpha\neq 1$ (with one-dimensional ideals $\langle E_1+E_2\rangle$, $\langle E_1-E_2\rangle$, $\mathfrak{g}_{1}$), b) $\mathfrak{g}_{4,5}^{\alpha,\beta}$, $-1<\alpha<\beta<1$, $\alpha\beta\neq 0$ or 
$\alpha=-1$, $0<\beta<1$ (with one-dimensional ideals $\langle E_1\rangle$, $\langle E_2\rangle$, $\langle E_3\rangle$).

In the case 3), a) the ideals  $\langle E_1\rangle$ and  $\mathfrak{g}_1$ aren't equivalent because
$\langle E_1\rangle\subset\mathfrak{g}'$ and $\mathfrak{g}_1\not\subset\mathfrak{g}'.$ The case 3), b) is similar.
In the case 3), c) the operator $\ad(-E_4):\mathfrak{g}^{\prime}\rightarrow\mathfrak{g}^{\prime}$ has its own subspaces 
$\langle E_1\rangle$ and $\langle E_2\rangle$ with different eigenvalues  $\alpha,$  $1$ for $\mathfrak{g}_{4,2}^{\alpha}$, $\alpha\notin\{0,1\}$ and $1,$ $0$ for $\mathfrak{g}_{4,3}.$ Denoting by $\mathfrak{g}$ any of these Lie algebras, by $\lambda_i$
eigenvalues for $\langle E_i\rangle,$ $i=1,2,$ and $\ad(-E_4),$ we get for every automorphism $\xi$ of the Lie algebra $\mathfrak{g}$:
$$[\xi(E_i),\xi(E_4)]=\xi([E_i,E_4])=\xi(\lambda_iE_i)=\lambda_i(\xi(E_i)).$$
Then $\xi(E_4)=E_4+f,$ $f\in\mathfrak{g}^{\prime}$ and $\xi(\langle E_i\rangle)=\langle E_i\rangle,$ $i=1,2;$ $\langle E_1\rangle$ and 
$\langle E_2\rangle$ are not equivalent.  

A similar argument is applicable to prove the part 4). In the case a), $\mathfrak{g}_1\not\subset\mathfrak{g}^{\prime}$ and the operator
$\ad E_3$ has eigenvectors $E_1+E_2$ and $E_1-E_2$, corresponding respectively to different eigenvalues $\alpha-1$ and $\alpha+1$.
Therefore one-dimensional ideals  $\langle E_1+E_2\rangle$, $\langle E_1-E_2\rangle$, $\mathfrak{g}_1$ are pairwise not equivalent.
In the case b), the operator $\ad(-E_4)$ has eigenvectors $E_1$, $E_2,$ and $E_3$, with corresponding different eigenvalues
$1$, $\beta$ and $\alpha$. Therefore, one-dimensional ideals $\langle E_1\rangle$, $\langle E_2\rangle$, $\langle E_3\rangle$  
are pairwise not equivalent.

According to Table 2 in paper \cite{Patera}, for every one-dimensional ideal $\mathfrak{L}$ of the Lie algebra $\mathfrak{g}$ from p.~2 of Proposition \ref{equiv} there exists a three-dimensional subspace $\mathfrak{q}$, containing $\mathfrak{L}$ and generating $\mathfrak{g}$. 
On the ground of Proposition \ref{equiv0}, any two such subspaces are equivalent. 
Besides, due to Lemma \ref{base}, the subspace $\mathfrak{q}$ cannot contain a one-dimensional ideal of the Lie algebra $\mathfrak{g}$ other than $\mathfrak{L}$.
This immediately implies that three-dimensional subspaces in $\mathfrak{g}$, generating $\mathfrak{g}$ and containing nonequivalent one-dimensional ideals, are not equivalent themselves.

If a three-dimensional subspace $\mathfrak{q}$, generating the Lie algebra $\mathfrak{g}$ from p. 2 of Proposition \ref{equiv}, contains no one-dimensional ideal of this algebra, then by Proposition \ref{ideal} and Remark \ref{rem1} there exists a basis  $(e_1,e_2,e_3:=[e_1,e_2])$ in $\mathfrak{q}$, satisfying Lemma \ref{base}, i.e. two-dimensional  subspace with a basis $(e_1,e_2)$ generates the Lie algebra $\mathfrak{g}$. In \cite{BerZub21} was proved that any 
two such two-dimensional subspaces are equivalent, moreover, $[e_2,e_3]=0$ for the Lie algebras from p. 2 of Proposition
\ref{equiv} with three-dimensional  commutative ideal;
$C_{23}^1\neq 0$, $C_{23}^2=0$ for the Lie algebras $\mathfrak{g}_{4,7}$, $\mathfrak{g}_{4,9}^{\alpha}$, $\alpha\geq 0$, 
$\mathfrak{g}_{4,8}^{\alpha}$, $-1\leq\alpha<1$, $\alpha\neq 0$;  $C_{23}^1=C_{23}^2=0$ for the Lie algebras $\mathfrak{g}_{4,8}^0$ and $\mathfrak{g}_{4,10}$.
This and Proposition \ref{afterrem} imply that any two such three-dimensional subspaces $\mathfrak{q}$ are equivalent.
\end{proof}

\begin{corollary}
Let $k$ be a number of equivalence classes of three-dimensional subspaces, generating a four-dimensional real Lie algebra $\mathfrak{g}$.
	
1. If  $\mathfrak{g}=\mathfrak{g}_{3,4}^{\alpha}\oplus\mathfrak{g}_1$, $0\leq \alpha\neq 1$, then $k=3$ for $\alpha=0$ and $k=4$ for $\alpha\neq 0.$ 
	
2. If $\mathfrak{g}=\mathfrak{g}_{4,2}^{\alpha}$, $\alpha\neq 0$, then
$k=1$ for  $\alpha=1$ and   $k=3$ in other cases.
	
3. If $\mathfrak{g}=\mathfrak{g}_{4,5}^{\alpha,\beta}$, 
$-1<\alpha\leq\beta\leq 1$, $\alpha\beta\neq 0$, or $\alpha=-1$, $0<\beta\leq 1$, then $k=1$ for $\alpha=\beta\neq 1$ or $\beta=1$, $\alpha\neq 1$ and $k=4$ in other cases.
	
4. If $\mathfrak{g}=\mathfrak{g}_{4,8}^{\alpha}$, $-1\leq\alpha\leq 1$, then $k=1$ for $\alpha=1$ and $k=2$ for $\alpha\neq 1$.
\end{corollary}

The next theorem follows from \cite{BerZub21}, Table \ref{Tab:list},  Propositions \ref{nonstr}, \ref{ideal}, \ref{afterrem} and proofs of Propositions \ref{solv2} -- \ref{equiv}.

\begin{theorem}
\label{mmm5}
Let $(\mathfrak{g},[\cdot,\cdot])$ and $\mathfrak{q}\subset \mathfrak{g}$ be a four-dimensional Lie algebra and a three-dimensional subspace, generating $\mathfrak{g}$ by the Lie bracket $[\cdot,\cdot]$.
	
1. If $\mathfrak{g}=2\mathfrak{g}_{2,1}$ then two following cases are posible: 
	
1) $\mathfrak{q}$ contains a one-dimensional ideal of the Lie algebra $\mathfrak{g}$;
	
2) $\mathfrak{q}$ has a basis $(e_1,e_2,e_3=[e_1,e_2])$ such that $[e_1,e_2]=0$ and $0\neq [e_2,e_3]\parallel e_3$. 
	
2. If  $\mathfrak{g}$ is one of the Lie algebras $\mathfrak{g}_{3,2}\oplus\mathfrak{g}_{1}$, $\mathfrak{g}_{3,4}^{\alpha}\oplus\mathfrak{g}_{1}$, $0\leq\alpha\neq 1$, $\mathfrak{g}=\mathfrak{g}_{4,3}$,  then three following cases are posible: 
	
1) $\mathfrak{C(g)}\subset\mathfrak{q}$;
	
2) $\mathfrak{q}$ contains a one-dimensional noncentral ideal of the Lie algebra $\mathfrak{g}$;
	
3) $\mathfrak{q}$ has a basis  $(e_1,e_2,e_3=[e_1,e_2])$ such that $[e_2,e_3]=0$.
	
3. If $\mathfrak{g}$ is one of the Lie algebras $\mathfrak{g}_{3,5}^{\alpha}\oplus\mathfrak{g}_{1}$, $\alpha\geq 0$,
$\mathfrak{g}_{4,1}$, then two following cases are posible: 
	
1)  $\mathfrak{C(g)}\subset\mathfrak{q}$;
	
2) $\mathfrak{q}$  has a basis $(e_1,e_2,e_3=[e_1,e_2])$ such that $[e_2,e_3]=0$.
	
4. If $\mathfrak{g}=\mathfrak{g}_{3,6}\oplus\mathfrak{g}_{1}$ then three following cases are posible: 
	
1) $\mathfrak{g}_1\subset\mathfrak{q}$; 
	
2) $\mathfrak{q}$  has a basis $(e_1,e_2,e_3=[e_1,e_2])$ such that $0\neq [e_2,e_3]\parallel e_2$; 
	
3) $\mathfrak{q}$ has a basis  $(e_1,e_2,e_3=[e_1,e_2])$ such that $0\neq [e_2,e_3]\parallel e_1$.
	
5. If $\mathfrak{g}=\mathfrak{g}_{3,7}\oplus\mathfrak{g}_{1}$  then two following cases are posible: 
	
1) $\mathfrak{C(g)}\subset\mathfrak{q}$;
	
2) $\mathfrak{q}$  has a basis $(e_1,e_2,e_3=[e_1,e_2])$ such that $0\neq [e_2,e_3]\parallel e_1$.
	
6. If $\mathfrak{g}$ is one of the Lie algebras $\mathfrak{g}_{4,2}^{\alpha}$, $\alpha\notin\{0,1\}$, $\mathfrak{g}_{4,4}$,
$\mathfrak{g}_{4,5}^{\alpha,\beta}$, $-1<\alpha<\beta<1$, $\alpha\beta\neq 0$ or $\alpha=-1$, $0<\beta\leq 1$, $\mathfrak{g}_{4,6}^{\alpha,\beta}$,  $\alpha>0$, $\beta\in\mathbb{R}$, then two following cases are posible: 
	
1)  $\mathfrak{q}$ contains a one-dimensional noncentral ideal of the Lie algebra $\mathfrak{g}$;
	
2) $\mathfrak{q}$  has a basis $(e_1,e_2,e_3=[e_1,e_2])$ such that $[e_2,e_3]=0$.
	
7. If $\mathfrak{g}$ is one of the Lie algebras $\mathfrak{g}_{4,7}$, $\mathfrak{g}_{4,8}^{\alpha}$, $-1<\alpha<1$, $\alpha\neq 0$, 
$\mathfrak{g}_{4,9}^{\alpha,\beta}$, $\alpha\geq 0$, $\beta\in\mathbb{R}$, then two following cases are posible: 
	
1)  $\mathfrak{q}$ contains a one-dimensional noncentral ideal of the Lie algebra $\mathfrak{g}$;
	
2)  $\mathfrak{q}$ has a basis $(e_1,e_2,e_3=[e_1,e_2])$ such that $C_{23}^1\neq 0$, $C_{23}^2=0$.
	
8. If  $\mathfrak{g}=\mathfrak{g}_{4,8}^{-1}$ or $\mathfrak{g}=\mathfrak{g}_{4,9}^{0}$ then two following cases are posible: 
	
1) $\mathfrak{C(g)}\subset\mathfrak{q}$;
	
2) $\mathfrak{q}$ has a basis  $(e_1,e_2,e_3=[e_1,e_2])$ such that $C_{23}^1\neq 0$, $C_{23}^2=0$. 
	
9. If  $\mathfrak{g}=\mathfrak{g}_{4,8}^{0}$  then two following cases are posible: 
	
1) $\mathfrak{q}$ contains a one-dimensional noncentral ideal of the Lie algebra $\mathfrak{g}$;
	
2)  $\mathfrak{q}$  has a basis $(e_1,e_2,e_3=[e_1,e_2])$ such that $C_{23}^1=C_{23}^2=0$. 
\end{theorem}

The next theorem follows from \cite{BerZub21},  Proposition \ref{threed}, Corollary \ref{act}, the proofs of Propositions \ref{equiv0} и \ref{equiv},  and Theorems \ref{main1}, \ref{mmm5}.

\begin{theorem}
\label{mmm6}
Let $G,$ $\mathfrak{q}\subset\mathfrak{g},$ and $d$ be respectively a four-dimensional connected Lie group with Lie algebra $(\mathfrak{g},[\cdot,\cdot])$, a three-dimensional subspace, generating $\mathfrak{g}$ by the Lie bracket $[\cdot,\cdot]$, and arbitrary left-invariant quasimetric $G,$ defined by some seminorm $F$ on $\mathfrak{q}$. Then
	
1. Every abnormal extremal of the space $(G,d)$ is nonstrongly abnormal for the Lie algebras  $\mathfrak{g}=\mathfrak{g}_{2,1}\oplus 2\mathfrak{g}_{1}$, $\mathfrak{g}=\mathfrak{g}_{3,1}\oplus\mathfrak{g}_{1}$ and in the  cases
2,~1); 3,~1); 4,~1); 5,~1); 8,~1) of Theorem \ref{mmm5}. 
	
2. Every abnormal extremal of the space $(G,d)$ is strongly abnormal for the Lie algebras $\mathfrak{g}_{3,3}\oplus\mathfrak{g}_{1}$, $\mathfrak{g}_{4,2}^1$, $\mathfrak{g}_{4,5}^{\alpha,1}$, $\mathfrak{g}=\mathfrak{g}_{4,5}^{\alpha,\alpha}$, $-1\leq \alpha<1$, $\alpha\neq 0$, 
and in the  cases	1,~1); 2,~2); 4,~2);  6,~1); 7,~1); 9,~1) of Theorem \ref{mmm5}.
	
3. For the Lie algebra $\mathfrak{g}_{4,10}$ and in the cases 1,~2); 2,~3); 3,~2); 6,~2); 9,~2) of Theorem \ref{mmm5}, 
abnormal extremal (\ref{dif}) (and every its left shift)  of the space $(G,d)$ is nonstrongly abnormal if and only if $F_U(k(s),s,0)=1/F(0,s,0)$  for some $k(s)\in\mathbb{R}$.
	
4. In the cases 4,~3); 5,~2), 7,~2); 8,~2) of Theorem \ref{mmm5},
abnormal extremal (\ref{dif}) (and every its left shift) of the space $(G,d)$ is nonstrongly abnormal if and only if $F_U(0,s,0)=1/F(0,s,0)$.
\end{theorem}

\end{document}